\documentclass[10pt,a4paper]{amsart}
\usepackage[utf8]{inputenc}
\usepackage[pagebackref,colorlinks,linkcolor=red,citecolor=blue,urlcolor=blue,hypertexnames=false]{hyperref}
\usepackage{csquotes}

\usepackage{amsmath,amsthm,amssymb,amsopn}
\usepackage{amsrefs}
\usepackage{amscd}
\input xy
\xyoption{all}
\newdir{ >}{{}*!/-5pt/@{>}}
\usepackage{mathtools}
\let\overfence\overbrace 
\let\downfencefill\downbracefill 
\patchcmd{\overfence}{\downbracefill}{\downfencefill}{}{}
\patchcmd{\downfencefill}{\braceru \bracelu}{}{}{}

\usepackage{etoolbox}
\apptocmd{\sloppy}{\hbadness 10000\relax}{}{}

\newcommand{\comment}[1]{} 

\newcommand\n{\mathfrak{n}}

\def\N{\mathbb{N}} 
\def\Z{\mathbb{Z}} 
\def\Q{\mathbb{Q}} 
 
\def\C{\mathbb{C}}

\def\topdf{\texorpdfstring}
\theoremstyle{plain}
\newtheorem{thm}[equation]{Theorem}
\newtheorem{lem}[equation]{Lemma}
\newtheorem{coro}[equation]{Corollary} 
\newtheorem{prop}[equation]{Proposition}

\theoremstyle{definition}

\newtheorem{ex}[equation]{Example}

\theoremstyle{remark} 

 \newtheorem{rem}[equation]{Remark}
 \newtheorem{nota}[equation]{Notation}
\newtheorem{stan}[equation]{Assumption}

  \numberwithin{equation}{section}
\newtheorem*{ack}{Acknowledgements}

\newcommand{\cA}{\mathcal A}
\newcommand{\cB}{\mathcal B}
\newcommand{\cC}{\mathcal C}

\newcommand{\cE}{\mathcal E}
\newcommand{\cF}{\mathcal F}
\newcommand{\cG}{\mathcal G}

\newcommand{\cI}{\mathcal I}

\newcommand{\cK}{\mathcal K}

\newcommand{\cM}{\mathcal M}

\newcommand{\cO}{\mathcal O}
\newcommand{\cP}{\mathcal P}

\newcommand{\cS}{\mathcal S}
\newcommand{\cT}{\mathcal T}
\newcommand{\cU}{\mathcal U}

\newcommand{\cZ}{\mathcal Z}
\def\fA{\mathfrak{A}}
\def\fa{\mathfrak{a}}
\def\fB{\mathfrak{B}}

\def\fC{\mathfrak{C}}

\def\fF{\mathfrak{F}}

\def\fX{\mathfrak{X}}

\def\fW{\mathfrak{W}}
\newcommand{\BF}{\fB\fF}
\newcommand{\vBF}{\check{\BF}}

\def\can{\operatorname{can}}

\def\gr{\operatorname{gr}}

\def\ann{\operatorname{Ann}}

\def\alg{\mathrm{Alg}}

\newcommand{\aha}{{\alg_{\ell}}}
\newcommand{\ahas}{{{\rm Alg}^*_{\ell}}}
\newcommand{\lra}{\longrightarrow}
\newcommand{\iso}{\overset{\sim}{\lra}}

\newcommand{\onto}{\twoheadrightarrow}
\newcommand{\fib}{\overset{\sim}{\onto}}
\newcommand{\into}{\hookrightarrow}
\newcommand{\ol}{\overline}

\def\reg{\operatorname{reg}}
\def\sing{\operatorname{sing}}
\def\sink{\operatorname{sink}}
\def\inf{\operatorname{inf}}

\def\triqui{\vartriangleleft}
\def\mspan{\operatorname{span}}
\def\supp{\operatorname{supp}}
\def\inc{\operatorname{inc}}

\def\diag{\operatorname{diag}}

\def\ad{\operatorname{ad}}

\def\id{\operatorname{id}}

\newcommand{\coker}{{\rm Coker}}
\renewcommand{\ker}{{\rm Ker}}

\newcommand{\im}{\mathrm{Im}}
\newcommand{\dom}{\mathrm{Dom}}
\newcommand{\op}{\mathrm{op}}

\DeclareMathOperator*{\colim}{colim}

\def\tor{\operatorname{Tor}}

\newcommand{\bu}{\bullet}

\def\ho{\operatorname{Ho}}
\def\spt{\operatorname{Spt}}

\renewcommand{\top}{\operatorname{top}}
\makeatletter
\@namedef{subjclassname@2020}{\textup{2020} Mathematics Subject Classification}
\makeatother

\begin{document}
\hfuzz=22pt
\vfuzz=22pt
\hbadness=2000
\vbadness=\maxdimen

\author{Guillermo Cortiñas}

\title{Exel-Pardo algebras with a twist}
\address{Departamento de Matem\'atica/IMAS\\ Facultad de Ciencias Exactas y Naturales\\ Universidad de Buenos Aires\\ Ciudad Universitaria \\ (1428) Buenos Aires}

\email{gcorti@dm.uba.ar}
\begin{abstract} 
Takeshi Katsura associated a $C^*$-algebra $C^*_{A,B}$ to a pair of matrices $A\ge 0$ and $B$ of the same size with integral coefficients, and gave sufficient conditions on $(A,B)$ to be simple purely infinite (SPI). We call such a pair a KSPI pair. It follows from a result of Katsura that any separable $C^*$-algebra $\fA$ which is a cone of a map $\xi:C(S^1)^n\to C(S^1)^n$ in Kasparov's triangulated category $KK$ is $KK$-isomorphic to $C^*_{A,B}$ for some KSPI pair $(A,B)$. In this article we introduce, for the data of a commutative ring $\ell$, matrices $A,B$ as above and a matrix $C$ of the same size with coefficients in the group $\cU(\ell)$ of invertible elements, an $\ell$-algebra $\cO_{A,B}^C$, the \emph{twisted Katsura algebra} of the triple $(A,B,C)$. When $C$ is trivial we recover the Katsura $\ell$-algebra first considered by Enrique Pardo and Ruy Exel. We show that if $\ell\supset\Q$ is a field and $(A,B)$ is KSPI, then $\cO_{A,B}^C$ is SPI, and that any $\ell$-algebra which is a cone of a map $\xi:\ell[t,t^{-1}]^n\to \ell[t,t^{-1}]^n$ in the triangulated bivariant algebraic $K$-theory category $kk$ is $kk$-isomorphic to $\cO_{A,B}^C$ for some $(A,B,C)$ as above so that $(A,B)$ is KSPI.  Katsura $\ell$-algebras are examples of the Exel-Pardo algebras $L(G,E,\phi)$ associated to a group $G$ acting on a directed graph $E$ and a $1$-cocycle $\phi:G\times E^1\to G$. Similarly, twisted Katsura algebras are examples of the twisted Exel-Pardo $\ell$-algebras $L(G,E,\phi_c)$ we introduce in the current article; they are associated to data $(G,E,\phi)$ as above twisted by a $1$-cocycle 
 $c:G\times E^1\to \cU(\ell)$. The algebra $L(G,E,\phi_c)$ can be variously described by generators and relations, as a quotient of a twisted semigroup algebra, as a twisted Steinberg algebra, as a corner skew Laurent polynomial algebra, and as a universal localization of a tensor algebra. We use each of these guises of $L(G,E,\phi_c)$ to study its $K$-theoretic, regularity and (purely infinite) simplicity properties. For example we show that if $\ell\supset \Q$ is a field, $G$ and $E$ are countable and $E$ is regular, then $L(G,E,\phi_c)$ is simple whenever the Exel-Pardo $C^*$-algebra $C^*(G,E,\phi)$ is, and is SPI if in addition the Leavitt path algebra $L(E)$ is SPI. 
\end{abstract}
 
\thanks{CONICET researcher partially supported by grants UBACyT 20020220300206BA, PIP 11220200100423CO and PICT 2021-2021-I-A-00710}
\maketitle

\section{Introduction}

An \emph{Exel-Pardo tuple} $(G,E,\phi)$ consists of a (directed) graph $$E:E^1\overset{r}{\underset{s}{\rightrightarrows}} E^0$$ together with an action of $G$ by graph automorphisms and a $1$-cocycle $\phi:G\times E^1\to G$ satisfying $\phi(g,e)(v)=g(v)$ for all $g\in G$, $e\in E^1$ and $v\in E^0$. To an Exel-Pardo tuple
$(G,E,\phi)$ and a commutative ring $\ell$ one associates a $C^*$-algebra $C^*(G,E,\phi)$ and an $\ell$-algebra $L(G,E,\phi)$. For example the Katsura $C^*$-algebra $C^*_{A,B}$ associated to a pair of matrices $A\in\N_0^{m\times n}$ and $B\in \Z^{m\times n}$ with $m\le n$ such that 
\begin{equation}\label{intro:a0b0}
    A_{i,j}=0\Rightarrow B_{i,j}=0,
\end{equation}
and its purely algebraic counterpart, the $\ell$-algebra $\cO_{A,B}=\cO_{A,B}(\ell)$ are Exel-Pardo algebras, where $E=E_{A}$ is the graph with reduced incidence matrix $A$ and where both the action of $G=\Z$ and the cocycle $\phi$ are determined by $B$. Katsura showed that $C^*_{A,B}$ is separable, nuclear and in the UCT class and that its (topological) $K$-theory is completely determined by the kernel and cokernel of the matrices $I-A^t$ and $I-B^t$. He further proved that under certain conditions on $(A,B)$, which we call KSPI, $C^*_{A,B}$ is simple purely infinite. Moreover his results show that every Kirchberg algebra $\fA$ with finitely generated $K$-theory is isomorphic, as a $C^*$-algebra, to $C^*_{A,B}$ for some KSPI pair $(A,B)$.  

In the current paper we consider Exel-Pardo tuples as above further \emph{twisted} by a $1$-cocyle $c:G\times E^1\to \cU(\ell)$ with values in the invertible elements of the ground ring $\ell$, to which we associate an algebra $L(G,E,\phi_c)$, the \emph{twisted Exel-Pardo algebra}. As a particular case, we obtain twisted Katsura algebras $\cO_{A,B}^C$ where $A,B$ are as above and $C\in\cU(\ell)^{m\times n}$ is such that 
\begin{equation}\label{intro:a0c1}
A_{i,j}=0\Rightarrow C_{i,j}=1.
\end{equation}

We study several properties of twisted Exel-Pardo algebras in general and of twisted Katsura algebras in particular. We define $L(G,E,\phi_c)$ by generators and relations (Section \ref{subsec:lgefdefi}) and show it can variously be regarded as a twisted groupoid algebra (Proposition \ref{prop:lgec=stein}), a corner skew Laurent polynomial ring (Section \ref{sec:twistlaure}) and a universal localization (Lemma \ref{lem:pgestolge}). For example, using the twisted groupoid picture and building upon the results of \cites{ep,steintwist,nonhau,steinprisimp,steiszak} we obtain the following simplicity criterion. Recall that a vertex $v\in E^0$ is \emph{regular} if it emits a nonzero finite number of edges. We write $\reg(E)\subset E^0$ for the subset of regular vertices and call $E$ \emph{regular} if $\reg(E)=E^0$.

\begin{thm}\label{thm:introsimp}
Let $\ell$ be a field of characteristic zero and $(G,E,\phi_c)$ a twisted Exel-Pardo tuple with $G$ countable and $E$ countable and regular.  If the Exel-Pardo $C^*$-algebra
$C^*(G,E,\phi)$ is simple, then $L(G,E,\phi_c)$ is simple. If furthermore $L(E)$ is simple purely infinite, then so is $L(G,E,\phi_c)$.
\end{thm}
The above result specializes to twisted Katsura algebras as follows.
\begin{coro}\label{coro:introsimp}
Let $(A,B,C)$ be as above and $\ell$ a field of characteristic zero. Assume that $(A,B)$ is a KSPI pair. Then $\cO_{A,B}^C$ is simple purely infinite. 
\end{coro}
The description of $L(G,E,\phi_c)$ by generators and relations provides an algebra extension akin to the Cohn extension of a Leavitt path algebra or the Toeplitz extension of a graph algebra. It has the form
\begin{equation}\label{intro:cohnext}
0\to \cK(G,E,\phi_c)\to C(G,E,\phi_c)\to L(G,E,\phi_c)\to 0    
\end{equation}

We use this extension to study $L(G,E,\phi_c)$ in terms of the bivariant algebraic $K$-theory category $kk$. 
We show that 
$\cK(G,E,\phi_c)$ and $C(G,E,\phi_c)$ are canonically $kk$-isomorphic to the crossed products $\ell^{\reg(E)}\rtimes G$ and $\ell^{E^0}\rtimes G$ (Proposition 
\ref{prop:kkkge} and Theorem \ref{thm:cohnkk}). Because the canonical functor $j:
\ahas\to kk$ sends algebra extensions to distinguished triangles, we get that if $E^0$ is finite, the Cohn extension above gives rise to a distinguished triangle in $kk$
\begin{equation}\label{intro:tria1}
j(\ell^{\reg(E)}\rtimes G)\overset{f}{\lra }j(\ell^{E^0}\rtimes G)\lra j(L(G,E,\phi_c))    
\end{equation}
We explicity compute $f$ in the case of twisted Katsura algebras. In this case $G=\Z$ acts trivially on $E^0$ so for $L_1=\ell[t,t^{-1}]$, $\rtimes G=\otimes L_1$ above. 
Since the kernel of the evaluation map $\sigma=\ker(ev_1:L_1\to\ell)$ represents the suspension in $kk$, we have $j(L_1)=j(\ell)\oplus j(\ell)[-1]$ in $kk$. We show in \ref{thm:katkkh} that for the reduced incidence matrix $A\in \N_0^{\reg(E)\times E^0}$ and $B\in\Z^{\reg(E)\times E^0}$ and $C\in\cU(\ell)^{\reg(E)\times E^0}$ satisfying \eqref{intro:a0b0} and \eqref{intro:a0c1} the map of \eqref{intro:tria1} has the following matricial form
\begin{equation}\label{intro:mapf}
\xymatrix{f:j(\ell)^{\reg(E)}\oplus(j(\ell)[-1])^{\reg(E)}\ar[rrr]^(.55){\begin{bmatrix} I-A^t& C^*\\ 0& I-B^t\end{bmatrix}}&&&j(\ell)^{E^0}\oplus j(\ell)[-1]^{E^0}}
\end{equation}
Here $C^*_{v,w}=C_{w,v}^{-1}$.
Recall from \cite{ct} that $\hom_{kk}(j(\ell)[-n],j(\ell))=KH_n(\ell)$ is Weibel's homotopy algebraic $K$-theory. The coefficients of $I-A^t$, $I-B^t$ and $C^*$ are regarded as elements of $kk(\ell,\ell)=KH_0(\ell)$ and $kk(\sigma,\ell)=KH_1(\ell)$ via the canonical maps $\Z\to KH_0(\ell)$ and $\cU(\ell)\to KH_1(\ell)$, which are isomorphisms when $\ell$ is a field or a PID, for example. In the latter cases we also have $kk(\ell,\sigma)=KH_{-1}(\ell)=K_{-1}(\ell)=0$; thus any element of 
$kk(L_1^m,L_1^n)$ is represented by a matrix with a zero block in the lower left corner. In fact we prove

\begin{thm}\label{thm:intromain} 
Let $\ell$ be a field or a PID, $n\ge 1$, and let $R\in\aha$ such that there is a distinguished triangle
\begin{equation}\label{intro:kat}
\xymatrix{j(\ell)^n\oplus j(\ell)[-1]^n\ar[rr]^{g}&& j(\ell)^n\oplus j(\ell)[-1]^n\ar[r]& j(R)}.
\end{equation}
Then there exist matrices $A\in M_{2n}(\N_0)$, $B\in M_{2n}(\Z)$ and $C\in M_{2n}(\cU(\ell))$ satisfying \eqref{intro:a0b0} and \eqref{intro:a0c1} with $(A,B)$ KSPI and an isomorphism
\[
j(\cO_{A,B}^C)\cong j(R).
\] 
\end{thm}

Putting Theorems \ref{thm:introsimp} and \ref{thm:intromain} together we get that if $\ell$ is a field of characteristic zero, then any $\ell$-algebra $R$ admitting a presentation \eqref{intro:kat} is $kk$-isomorphic to a simple purely infinite twisted Katsura algebra. An analogous result in the Kasparov $KK$-category of $C^*$-algebras follows from work of Katsura \cite{kat}*{Proposition 3.2} showing that any $C^*$-algebra with a $KK$-presentation of the form \eqref{intro:kat} is $KK$-isomorphic to a KSPI untwisted Katsura $C^*$-algebra. This uses the fact that for Kasparov's $KK$ and Banach algebraic $K$-theory, we have $KK(\C[-1],\C)=K_1^{\top}(\C)=0$. Note however that in the purely algebraic context, it is useful to introduce the twist $C$, since as $kk(\ell[-1],\ell)=\ell^*$ is nonzero, the matricial form of the map $g$ in \eqref{intro:kat} need not have a trivial block in the upper right corner.  

\goodbreak

Observe that applying $KH_*(-)$ to the triangle \eqref{intro:tria1} gives a long exact sequence
\begin{multline}\label{intro:longseq}
KH_{n+1}(L(G,E,\phi_c))\to KH_n(\ell^{(\reg(E))}\rtimes G)\overset{f}{\lra}\\
KH_n(\ell^{(E^0)}\rtimes G)\to KH_{n}(L(G,E,\phi_c))
\end{multline}
One may ask whether a similar sequence holds for Quillen's $K$-theory. This will be the case if every ring $R$ appearing in \eqref{intro:longseq} is $K$-regular, for in this case the comparison map $K_*(R)\to KH_*(R)$ is an isomorphism. In Section \ref{sec:twistlaure} we observe that when $E$ is finite without sources, the $\Z$-graded algebra $L(G,E,\phi_c)$ can be regarded as a Laurent polynomial ring twisted by certain corner isomorphism $\psi$ of the homogeneous component of degree zero. We use this to give a sufficient condition for the $K$-regularity of $L(G,E,\phi_c)$ where $E$ is any row-finite graph and $G$ acts trivially on $E^0$; 
see Theorem \ref{thm:Kreg}. As a particular case we obtain that if $\ell[G]$ is regular supercoherent and the untwisted Exel-Pardo tuple $(G,E,\phi)$ is pseudo-free in the sense of \cite{ep}*{Definition 5.4}, then $L(G,E,\phi_c)$ is $K$-regular.     
For the specific case of twisted Katsura algebras, we have

\begin{prop}\label{intro:KatKreg}
Let $\ell$ be a field and $(A,B,C)$ a twisted Katsura triple. If either of the following holds  then $\cO_{A,B}^C$ is $K$-regular. 
\item[i)] $B_{v,w}=0\iff A_{v,w}=0$.
\item[ii)] If $v\in\reg(E)$ is such that $B_{v,w}=0$ for some $w\in r(s^{-1}\{v\})$, then $B_{v,w'}=0$ for all $w'\in r(s^{-1}\{v\})$. 
\end{prop}

In Section \ref{sec:uniloc} we show that if $E$ is finite, then $L(G,E,\phi_c)$ can be described as a universal localization of certain tensor algebra (Lemma \ref{lem:pgestolge}), much in the same way as the Leavitt path algebra $L(E)$ is a universal localization of the usual path algebra $P(E)=T_{\ell^{E^0}}(\ell^{E^1})$. 
We use this to show, in Proposition \ref{prop:lgereg}, that under rather mild conditions, $L(G,E,\phi_c)$ is a regular ring, in the sense that every module over it has finite projective dimension. In particular, we get

\begin{prop}\label{prop:introKatreg}
Let $\ell$ be field and $(A,B,C)$ finite square matrices satisfying \eqref{intro:a0b0} and \eqref{intro:a0c1}. Then $\cO_{A,B}^C$ is a regular ring.
\end{prop} 

The rest of this article is organized as follows. Section \ref{sec:prelis} starts by introducing some notation and recalling basic facts on Exel-Pardo tuples (Subsections \ref{subsec:alg} and \ref{subsec:eptuples}). Let $\cP(E)$ be the set of all finite paths in a graph $E$. Lemma \ref{lem:epextend} recalls that if $(G,E,\phi)$ is an $EP$-tuple, then the $G$-action on $E$ extends to an action on $\cP(E)$ and that $\phi$ extends to a $1$-cocycle $G\times\cP(E)\to G$. Subsection \ref{subsec:twistep} introduces twisted $EP$-tuples.  Lemma \ref{lem:twistextend} shows that there is an essentially unique way to extend a $1$-cocycle $c:G\times E^1\to G$ to a $1$-cocycle $G\times \cP(E)\to \cU(\ell)$ that is compatible with $\phi$ and with the $G$-action. In Subsection \ref{subsec:twistsemi} we recall the Exel-Pardo pointed inverse semigroup $\cS(G,E,\phi)$ and show that $c$ induces a semigroup $2$-cocycle $\omega:\cS(G,E,\phi)^2\to\cU(\ell)_\bu=\cU(\ell)\cup\{0\}$. Section \ref{sec:viacohn} is devoted to introducing all the algebras of the extension \eqref{intro:cohnext}.
In Subsection \ref{subsec:cohn} we define $C(G,E,\phi_c)$ by generators and relations and prove in Proposition \ref{prop:cohnpres} that it is isomorphic to the twisted semigroup algebra $\ell[\cS(G,E,\phi),\omega]$. This automatically gives an $\ell$-linear basis $\cB$ for $C(G,E,\phi_c)$ (Corollary \ref{coro:cohnpres}). In Subsection \ref{subsec:kgecp} we define $\cK(G,E,\phi_c)$ as the two-sided ideal of $C(G,E,\phi_c)$ generated by certain elements and prove in Proposition \ref{prop:kge} that it is isomorphic to the crossed-product of $G$ with certain ultramatricial algebra; this again gives an $\ell$-linear basis $\cB'$ of $\cK(G,E,\phi_c)$ (see \eqref{B'}). The next subsection introduces a subset $\cB"\subset\cB$ such that, under certain hypothesis,  $\cB'\cup\cB"$ is an $\ell$-linear basis for $C(G,E,\phi_c)$ (Proposition \ref{prop:basisunion}). The hypothesis are satisfied both when the $G$-action on $E$ and the $1$-cocycle $\phi$ are trivial, and when the action satisifies a weak version of the notion of pseudo-freeness introduced in \cite{ep}*{Definition 5.4} which we call partial pseudo-freeness.
Subsection \ref{subsec:lgefdefi} defines $L(G,E,\phi_c)$ by the extension \eqref{intro:cohnext}. Then we show that the  canonical map $L(E)\to L(G,E,\phi_c)$ from the Leavitt path algebra is injective (Proposition \ref{prop:lesubsetlge}) and use  Proposition \ref{prop:basisunion} to give a linear basis for $L(G,E,\phi_c)$ in Corollary \ref{coro:basisunion} under the hypothesis of that proposition.
Section \ref{sec:epstei} describes $L(G,E,\phi_c)$ as a twisted groupoid algebra in the sense of \cites{steintwist,resteintwist} and uses this description to give simplicity criteria. Subsection \ref{subsec:twistgpd} deals with the general set-up of an inverse semigroup $\cS$ acting on a space, explains how to go from a semigroup 2-cocycle $\nu:\cS\times \cS\to \cU(\ell)_\bu$ to a groupoid $2$-cocycle $\tilde{\nu}:\cG\times\cG\to \cU(\ell)$ on the groupoid $\cG$ of germs, and expresses the twisted Steinberg algebra $\cA(\cG,\tilde{\nu})$
as a quotient of the twisted semigroup algebra $\ell[\cS,\nu]$ (Lemma \ref{lem:ellStostein}). Subsection \ref{subsec:twistgpd} applies the above to the cocycle $\omega:\cS(G,E,\phi)^2\to\cU(\ell)_\bu$ and the tight groupoid $\cG(G,E,\phi_c)$ and shows that $L(G,E,\phi_c)\cong \cA(\cG(S(G,E,\phi_c),\tilde{\omega})$ (Proposition \ref{prop:lgec=stein}). 
Theorem \ref{thm:introsimp} is proved as Theorem \ref{thm:spistein}. Section \ref{sec:kat} introduces twisted Katsura algebras. Corollary \ref{coro:introsimp} is proved as Theorem \ref{thm:katpis}. A version of the latter theorem valid over fields of arbitrary characteristic, but with additional hypothesis on the matrices $A$ and $B$ is proved in Proposition
\ref{prop:katsimp}. Section \ref{sec:kk} is concerned with bivariant algebraic $K$-theory. Subsection \ref{subsec:prelikk} recalls some basic facts about $kk$. The next two subsections are concerned with the algebras $\cK(G,E,\phi_c)$ and $C(G,E,\phi_c)$ in the extension \eqref{intro:cohnext}. Proposition \ref{prop:kkkge} shows that $\cK(G,E,\phi_c)$ is $kk$-isomorphic to $\ell^{(\reg(E))}\rtimes G$ and Theorem \ref{thm:cohnkk} that $C(G,E,\phi_c)$ is $kk$ -isomorphic to $\ell^{(E^0)}\rtimes G$. Section \ref{sec:kkat} is about twisted Katsura algebras in $kk$; the map $f$ of \eqref{intro:tria1} is computed in this section; see Theorem \ref{thm:katkkh} and Corollary \ref{coro:katkkh}. A short exact sequence computing homotopy $K$-theory $KH_*(\cO_{A,B}^C)$ is obtained in Corollary \ref{coro:katkh}. Theorem \ref{thm:intromain} is a particular case of Theorem \ref{thm:katkat}; see Corollary \ref{coro:katkat}. Section \ref{sec:twistlaure} is concerned with $K$-regularity of $L(G,E,\phi_c)$. It starts by observing that if $E$ is finite without sources, 
$L(G,E,\phi_c)$ can be regarded as a corner skew Laurent polynomial ring in the sense of \cite{fracskewmon}. This is then used in Theorem \ref{thm:Kreg} to give a general criterion for the $K$-regularity of the algebra of a twisted $EP$-tuple
$(G,E,\phi_c)$ where $E$ is row-finite and $G$ acts trivially on $E^0$. It applies in particular if $(G,E,\phi_c)$ is pseudo-free and $\ell[G]$ is regular supercoherent (Corollary \ref{coro:kregpseudo}). The next section investigates $K$-regularity of twisted Katsura algebras; Proposition \ref{prop:introKatreg} is proved as Proposition \ref{prop:KatKreg}.
Finally in Section \ref{sec:uniloc}, still under the assumption that $G$ acts trivially on $E^0$, we describe $L(G,E,\phi_c)$ as a universal localization of a certain tensor algebra (Lemma \ref{lem:pgestolge}) and use this to show in Proposition \ref{prop:lgereg} that under rather mild conditions, $L(G,E,\phi_c)$ is a regular ring. In particular,
$\cO_{A,B}^C$ is regular (Corollary \ref{coro:lgereg}).

\begin{ack}
I wish to thank Becky Armstrong, Guido Arnone and Enrique Pardo for carefully reading an earlier draft of this article and making useful comments. Thanks also to Marco Farinati for a fruitful discussion leading to Lemmas \ref{lem:hhreg} and \ref{lem:treg}. 
\end{ack}

\section{Preliminaries}\label{sec:prelis}

\subsection{Algebras}\label{subsec:alg}

We fix a commutative unital ring $\ell$. 
By an \emph{algebra} we shall mean a symmetric $\ell$-bimodule $A$ together with an associative product $A\otimes_\ell A\to A$. For a set $X$ and an algebra $R$, let 
\begin{gather*}
\Gamma_X^wR=\{A:X\times X\to R: |\supp(A(x,-))|<\infty>|\supp(A(-,x))|,\ \ \forall x\in X\}\\
\Gamma_X^aR=\{A\in\Gamma^w_X R: \exists N\in\N \ \ |\supp(A(x,-))|,|\supp(A(-,x))|\le N\ \ \forall x\in X\}\\
\Gamma_XR=\{A\in\Gamma_X^aR: |\im(A)|<\infty\}.
\end{gather*}
We regard an element of any  of the sets above as an $X\times X$ matrix; observe that matricial sum and product make all three sets above into $\ell$-algebras. When $R=\ell$, we drop it from the notation; thus $\Gamma_X=\Gamma_X(\ell)$, etc. Let $\cI(X)$ be the inverse semigroup of all partially defined injections $X\supset \dom(\sigma)\overset{\sigma}{\lra} X$. If $\sigma\in\cI(X)$ and $\mathrm{Graph}(\sigma)\subset X\times X$ is its graph, then its characteristic function lives in $\Gamma_X$
\[
U_f:=\chi_{\mathrm{Graph}(\sigma)}\in\Gamma_X.
\]
Let $R^X$ be the set of all functions $X\to R$. For $a\in R^X$, let 
\[
\diag(a)_{x,y}=\delta_{x,y}a_x.
\] 
Observe that for all $a\in R^X$ and $\sigma\in\cI(X)$, 
\[
\Gamma^a_XR\owns \diag(a)U_\sigma.
\]
If $a$ takes finitely many distinct values, then $\diag(a)U_\sigma\in\Gamma_XR$; in fact by \cite{hcsmall}*{Lemma 2.1}
 the latter elements generate $\Gamma_XR$ as an $\ell$-module, at least when $X$ is countable.

\subsection{Exel-Pardo tuples}\label{subsec:eptuples}
\numberwithin{equation}{subsection}
Let $G$ be a group and $X$ a set, together with a $G$-action
\[
G\times X\to X,\ \ (g,x)\mapsto g(x).
\]
Let $H$ be a group; a \emph{$1$-cocycle} with values in $H$ for the $G$-set $X$ is a function $\psi:G\times X\to H$ such that, for every $g,h\in G$ and $x\in X$, we have
\[
\psi(gh,x)=\psi(g,h(x))\psi(h,x).
\]
A (directed) \emph{graph} $E$ consists of sets $E^0$ and $E^1$ of vertices and edges and \emph{range} and \emph{source} maps $r,s:E^1\to E^0$. If $e\in E^1$, then $s(e)$ and $r(e)$ are respectively the \emph{source} and the \emph{range} of $e$. Let $n\ge 1$; a sequence of edges $\alpha=e_1\cdots e_n$ such that $r(e_i)=s(e_{i+1})$ for all $1\le i\le n-1$ is called a \emph{path of length} $|\alpha|=n$, with source $s(\alpha)=s(e_1)$ and range $r(\alpha)=r(e_n)$. Vertices are regarded as paths of length $0$. We write $\cP(E)$ for the set of all paths of finite length, and, abusing notation, also for the graph with vertices $E^0$, edges $\cP(E)$ and range and source maps as just defined. The set $\cP(E)$ is partially ordered by
\begin{equation}\label{pathorder}
\alpha\ge\beta\iff (\exists \gamma)\, \beta=\alpha\gamma.     
\end{equation}

If $E$ and $F$ are graphs, then by a homomorphism $h:E\to F$ we understand a pair
of functions $h^{i}:E^{i}\to F^{i}$, $i=0,1$ such that $r\circ h^{1}=h^{0}\circ r$ and $s\circ h^{1}=h^{0}\circ s$. A \emph{$G$-graph} is a graph $E$ equipped with an action of $G$ by graph automorphisms. An \emph{Exel-Pardo $1$-cocycle} for $E$ with values in $G$ is a $1$-cocycle $\phi:G\times E^1\to G$ such that 
\begin{equation}\label{eq:epcoci}
\phi(g,e)(v)=g(v)
\end{equation}
 for all $g\in G$, $e\in E^1$ and $v\in E^0$. An \emph{Exel-Pardo tuple} is a tuple $(G,E,\phi)$ consisting of a group $G$, a $G$-graph $E$ and a $1$-cocycle as above. 

\begin{lem}[\cite{ep}*{Proposition 2.4}]\label{lem:epextend}
Let $(G,E,\phi)$ be an Exel-Pardo tuple. Then the $G$-action on $E$ and the cocycle $\phi$ extend respectively to a
$G$-action and a $1$-cocycle on the path graph $\cP(E)$ satisfying all four conditions below.
\item[i)] $\phi(g,v)=g$ for all $v\in E^0$.
\item[ii)] $|g(\alpha)|=|\alpha|$ for all $\alpha\in\cP(E)$. 
The next two conditions hold for all concatenable $\alpha$, $\beta\in \cP(E)$. 
\item[iii)] $g(\alpha\beta)=g(\alpha)\phi(g,\alpha)(\beta)$ 
\item[iv)] $\phi(g,\alpha\beta)=\phi(\phi(g,\alpha),\beta)$. 

\goodbreak

Moreover, such an extension is unique.
\end{lem}

\subsection{Twisted Exel-Pardo tuples}\label{subsec:twistep}

Write $\cU(\ell)$ for the group of inverible elements of our ground ring $\ell$.  A \emph{twisted Exel-Pardo tuple} is an Exel-Pardo tuple $(G,E,\phi)$ together with a $1$-cocyle $c:G\times E^1\to \cU(\ell)$. Remark that 
\[
\phi_c:G\times E^1\to \cU(\ell)G\subset \cU(\ell[G]), \ \phi_c(g,e)=c(g,e)\phi(g,e)
\]
is a $1$-cocycle with values in the multiplicative group of the group algebra $\ell[G]$. We write $(G,E,\phi_c)$ for the twisted EP-tuple $(G,E,\phi,c)$. 

\begin{lem}\label{lem:twistextend} Let $(G,E,\phi_c)$ be a twisted Exel-Pardo tuple. Then $c$ extends uniquely to a $1$-cocycle $G\times\cP(E)\to \cU(\ell)$ satisfying
\begin{equation}\label{eq:conditwist}
c(g,v)=1,\,\, c(g,\alpha\beta)=c(g,\alpha)c(\phi(g,\alpha),\beta)
\end{equation}
for all concatenable $\alpha,\beta\in\cP(E)$. 
\end{lem}
\begin{proof} The prescriptions \eqref{eq:conditwist} together with Lemma \ref{lem:epextend} dictate that we must have
\begin{equation}\label{map:cext}
c(g,v)=1,\ \ c(g,e_1\cdots e_n)=c(g,e_1)\prod_{i=1}^{n-1}c(\phi(g,e_1\cdots e_i),e_{i+1})
\end{equation}
for every vertex $v\in E^0$ and every path $e_1\cdots e_n\in\cP(E)$. 
We have to check that the formulas \eqref{map:cext} define a $1$-cocycle satisfying \eqref{eq:conditwist}. It is clear that \eqref{eq:conditwist} is satisfied by \eqref{map:cext} whenever $\alpha$ or $\beta$ are vertices. Let $n,m\ge 1$ and let $e_1\cdots e_{n+m}\in\cP(E)$. Then using \eqref{map:cext} and using part iv) of Lemma \ref{lem:epextend} at the second identity, we have 
\begin{align*}
c(\phi(g,e_1\dots e_n),e_{n+1}\cdots e_{n+m})=& c(\phi(g,e_1\dots e_n), e_{n+1})\cdot\\
&\prod_{j=1}^{m-1}c(\phi(\phi(g,e_1\dots e_n),e_{n+1}\cdots e_{n+j}),e_{n+j+1})\\
=& \prod_{j=0}^{m-1}c(\phi(g,e_1\dots e_{n+j}), e_{n+j+1}).
\end{align*}
Hence using \eqref{map:cext} and the identity we have just proved, we obtain
\begin{align*}
c(g,e_1\cdots e_n)c(\phi(g,e_1\dots e_n),e_{n+1}\cdots e_{n+m})=& c(g,e_1)\prod_{i=1}^{n+m-1}c(\phi(g,e_1\cdots e_i),e_{i+1})\\
=&c(g,e_1\cdots e_{n+m}). 
\end{align*}
We have thus established that \eqref{map:cext} satisfies \eqref{eq:conditwist}; it remains to show that it is a $1$-cocycle. Let $g,h\in G$, $n\ge 1$, and $e_1\cdots e_n\in\cP(E)$. Then by (iii) of Lemma \ref{lem:epextend}
\[
h(e_1\cdots e_n)=h(e_1)\phi(h,e_1)(e_2)\cdots \phi(h,e_{1}\cdots e_{n-1})(e_n).
\]
Hence using the above identity and the facts that $\phi$ is a cocycle on $\cP(E)$ and that $c$ is a cocycle on $E^1$ we obtain
\begin{gather*}
c(gh,e_1\dots e_n)= c(gh,e_1)\prod_{i=1}^{n-1}c(\phi(gh,e_1\cdots e_i),e_{i+1})\\
                  = c(g,h(e_1))c(h,e_1)\prod_{i=1}^{n-1}c(\phi(g,h(e_1\cdots e_i))\phi(h,e_1\cdots e_i),e_{i+1})\\
									=\left(c(h,e_1)\prod_{i=1}^{n-1}c(\phi(h,e_1\cdots e_i),e_{i+1})\right)\cdot\\ 
									\left(c(g,h(e_1))\prod_{i=1}^{n-1}c(\phi(g,h(e_1)\phi(h,e_1)(e_2)\cdots \phi(h,e_{1}\cdots e_{i-1})(e_i)),\phi(h,e_{1}\cdots e_{i})(e_{i+1}))\right)\\
									=c(h,e_1\cdots e_n)c(g,h(e_1)\phi(h,e_1)(e_2)\cdots \phi(h,e_1\dots e_{n-1}(e_n))\\
									= c(h,e_1\cdots e_n)c(g, h(e_1\cdots e_n)).
\end{gather*}
\end{proof}
\begin{rem}\label{rem:phic}
Lemma \ref{lem:twistextend} together with part (iv) of Lemma \ref{lem:epextend} imply that, for extension of $\phi_c$ to a map $\ell[G]\times\cP(E)\to \ell[G]$, $\ell$-linear on the first variable, we have
\[
\phi_c(g,\alpha\beta)=\phi_c(\phi_c(g,\alpha),\beta).
\]
\end{rem}
\subsection{Twists and semigroups}\label{subsec:twistsemi}
Let $\cS$ be an inverse semigroup with inverse operation $s\mapsto s^*$, pointed by a zero element $\emptyset$. A (normalized) \emph{$2$-cocycle} with values in the pointed inverse semigroup $\cU(\ell)_\bu=\cU(\ell)\cup\{0\}$ is a map $\omega:\cS\times\cS\to \cU(\ell)_\bu$ such that for all $s,t,u\in\cS$ and $x\in\cS\setminus\{\emptyset\}$, 
\begin{gather}\label{eq:2cocy1}
\omega(st,u)\omega(s,t)=\omega(s,tu)\omega(t,u),\\ 
\omega(\emptyset,s)=\omega(s,\emptyset)=0,\ \ \omega(x,x^*)=\omega(xx^*,x)=\omega(x,x^*x)=1.\label{eq:2cocy2}
\end{gather}
The \emph{twisted semigroup algebra} $\ell[\cS,\omega]$ is the $\ell$-module $\ell[\cS]=(\bigoplus_{s\in S}\ell s)/\ell \emptyset$ equipped with the $\ell$-linear multiplication $\cdot_\omega$ induced by
\begin{equation}\label{semitwist}
s\cdot_\omega t=st\omega(s,t).    
\end{equation}
A straightforward calculation shows that $\ell[\cS,\omega]$ is an associative algebra and that for every $s\in \cS$,
\begin{equation*}
s\cdot_\omega s^*\cdot_\omega s=s.
\end{equation*}

Let $(G,E,\phi_c)$ be a twisted Exel-Pardo tuple. Recall from \cite{ep}*{Definition 4.1} that there is a pointed inverse semigroup
\[
\cS(G,E,\phi)=\{(\alpha,g,\beta)\colon \alpha,\beta\in\cP(E), \thickspace r(\alpha)=g(r(\beta))\}\cup\{0\}.
\]
Multiplication in $\cS(G,E,\phi)$ is defined by

\begin{equation}\label{prod:sge}
(\alpha,g,\beta)(\gamma,h,\theta)=\left\{\begin{matrix}(\alpha g(\gamma_1),\phi(g,\gamma_1)h,\theta)& \text{ if } \gamma=\beta\gamma_1\\
                                                       (\alpha, g\phi(h,h^{-1}(\beta_1)),\theta h^{-1}(\beta_1))&\text{ if }\beta=\gamma\beta_1\\
																											0& \text{ otherwise.}\end{matrix}\right.
\end{equation}
The inverse of an element of $\cS(G,E,\phi)$ is defined by $(\alpha,g,\beta)^*=(\beta,g^{-1},\alpha)$. 
Define a map 
\begin{gather}\label{map:omega}
\omega:\cS(G,E,\phi)\times \cS(G,E,\phi)\to \cU(\ell)_\bu\\
\omega(0,s)=\omega(s,0)=0\ \ \forall s\in \cS(G,E,\phi),\nonumber\\
\omega((\alpha,g,\beta),(\gamma,h,\theta))=\left\{\begin{matrix}\\
c(h,h^{-1}(\beta_1))& \text{ if } \beta=\gamma\beta_1\\ c(g,\gamma_1)& \text{ if } \gamma=\beta\gamma_1\\ 
0& \text{ otherwise.}\end{matrix}\right. \nonumber
\end{gather}

\begin{lem}\label{lem:omega}
The map $\omega$ of \eqref{map:omega} is a $2$-cocycle. 
\end{lem}
\begin{proof} It is clear that $\omega$ satisfies the identities \eqref{eq:2cocy2}. We have to check that the identity \eqref{eq:2cocy1} holds for all $s,t,u\in\cS(G,E,\phi)$. If any of $s,t,u$ is $0$, this is clear. Observe that for nonzero $s,t$, $c(s,t)=0$ exactly when $st=0$. Hence \eqref{eq:2cocy1} is also clear whenever $st=0$ or $tu=0$. To check the remaining cases, put $s=(\alpha,g,\beta)$, $t=(\gamma,h,\theta)$, $u=(\xi, k, \eta)$. 
\begin{enumerate}
\item $\beta=\gamma\beta_1$, $\xi=\theta\xi_1$. Then
\[
st=(\alpha, g\phi(h,h^{-1}(\beta_1)),\theta h^{-1}(\beta_1)),
 \ \ tu=(\gamma h(\xi_1), \phi(h,\xi_1)k,\eta).
\]
We divide this case in 3 subcases:
\begin{enumerate}
\item $\xi_1$ and $h^{-1}(\beta_1)$ are incomparable with respect to the path order \eqref{pathorder}. Then $ h(\xi_1)$ and $\beta_1$ are incomparable also, hence
$\omega(s,tu)=\omega(st,u)=0$, so \eqref{eq:2cocy1} holds in this case. 
\item $\xi_1=h^{-1}(\beta_1)\xi_2$. Then $h(\xi_1)=\beta_1\phi(h,h^{-1}(\beta_1))(\xi_2)$ and we have
\begin{align*}
\omega(st,u)\omega(s,t)=&c(g\phi(h,h^{-1}(\beta_1)),\xi_2)c(h,h^{-1}(\beta_1))\\
&=c(g,\phi(h,h^{-1}(\beta_1))(\xi_2))c(\phi(h,h^{-1}(\beta_1)),\xi_2)c(h,h^{-1}(\beta_1))\\
&=\omega(s,tu)c(h,h^{-1}(\beta_1)\xi_2)\\
&=\omega(s,tu)\omega(t,u).
\end{align*}
\item $h^{-1}(\beta_1)=\xi_1\beta_2$. Then $\beta_1=h(\xi_1)\phi(h,\xi_1)(\beta_2)$ and we have
\begin{align*}
\omega(s,tu)\omega(t,u)=&c(\phi(h,\xi_1)k, k^{-1}(\beta_2))c(h,\xi_1)\\
=&c(\phi(h,\xi_1),\beta_2)c(k,k^{-1}(\beta_2))c(h,\xi_1)\\
=&c(k,k^{-1}(\beta_2))c(h,\xi_1\beta_2)\\
=&\omega(st,u)\omega(s,t)
\end{align*}
\end{enumerate}
\item $\gamma=\beta\gamma_1$ and $\xi=\theta\xi_1$. Then
\[
st=(\alpha g(\gamma_1),\phi(g,\gamma_1)h,\theta),\ \ tu=(\gamma h(\xi_1),\phi(h,\xi_1)k,\eta)
\]
and we have
\begin{align*}
\omega(st,u)\omega(s,t)=&c(\phi(g,\gamma_1)h,\xi_1)c(g,\gamma_1)\\
=&c(\phi(g,\gamma_1),h(\xi_1))c(h,\xi_1)c(g,\gamma_1)\\
=&c(g,\gamma_1h(\xi_1))c(h,\xi_1)\\
=&\omega(s,tu)w(t,u).
\end{align*}

\item $\beta=\gamma\beta_1$, $\theta=\xi\theta_1$. Then 
\begin{align*}
\omega(st,u)\omega(s,t)=&c(k,k^{-1}(\theta_1 h^{-1}(\beta_1)))c(h,h^{-1}(\beta_1))\\
=&c(k,k^{-1}(\theta_1)(\phi(k^{-1},\theta_1)h^{-1})(\beta_1))c(h,h^{-1}(\beta_1))\\
=&c(k,k^{-1}(\theta_1))c(\phi(k, k^{-1}(\theta_1)), (\phi(k^{-1},\theta_1)h^{-1})(\beta_1))c(h,h^{-1}(\beta_1))\\
=&\omega(t,u)c(\phi(k, k^{-1}(\theta_1)),(\phi(k, k^{-1}(\theta_1))^{-1}h^{-1})(\beta_1))c(h,h^{-1}(\beta_1))\\
=&\omega(t,u)c(h\phi(k,k^{-1}(\theta_1)),(\phi(k,k^{-1}(\theta_1))^{-1}h^{-1})(\beta_1)))\\
=&\omega(t,u)\omega(s,tu).
\end{align*}

\item $\gamma=\beta\gamma_1$, $\theta=\xi\theta_1$. Then
\begin{align*}
\omega(st,u)\omega(s,t)=& c(k,k^{-1}(\theta_1))c(g,\gamma_1)\\
=&\omega(t,u)\omega(s,tu).
\end{align*}
\end{enumerate}
\end{proof}

\section{Twisted Exel-Pardo algebras via the Cohn extension}\label{sec:viacohn}
\subsection{Cohn algebras}\label{subsec:cohn}
The \emph{Cohn algebra} of the twisted Exel-Pardo tuple $(G,E,\phi_c)$ is the quotient $C(G,E,\phi_c)$ of the free algebra on 
the set $$\{v,\ vg,\ e, \ eg,\ e^*,\ ge^*:\ \  v\in E^0,\ \ g\in G,\  \ e\in E^1\}$$ modulo the following relations
\begin{gather}
v=v1,\ \ e=e1,\ \ 1e^*=e^*,\ eg=s(e)e(r(e)g),\ ge^*=(g(r(e))g)e^*s(e),  \label{eq:cohn0}\\
 e^*f=\delta_{e,f}r(e),\ (vg)wh=\delta_{v,g(w)}vgh, \label{eq:cohn1}\\
(vg)e=\delta_{v,g(s(e))}g(e)\phi_c(g,e),\ \ e^*vg=\delta_{v,s(e)}\phi_c(g,g^{-1}(e))(g^{-1}(e))^*.\label{eq:cohn2}
\end{gather} 
If $E^0$ happens to be finite, then $C(G,E,\phi_c)$ is unital with unit $1=\sum_{v\in E^0}v$, and $g\mapsto \sum_{v\in E^0}vg$ is a group homomorphism $G\to \cU(C(G,E,\phi_c))$. It will follow from Proposition \ref{prop:cohnpres} below that the latter is a monomorphism, and we will identify $G$ with its image in $\cU(C(G,E,\phi_c))$. Thus if $E^0$ is finite, then writing $\cdot$ for the product in $C(G,E,\phi_c)$, we have $vg=v\cdot g$ and $eg=e\cdot g$. However if $E^0$ is infinite, we do not identify $G$ with any subset of $C(G,E,\phi_c)$.   

\begin{prop}\label{prop:cohnpres}
Let $\omega$ be as in \eqref{map:omega}. There is an isomorphism of algebras $C(G,E,\phi_c)\iso \ell[\cS(G,E,\phi),\omega]$ mapping $vg\mapsto (v,g,g^{-1}(v))$, $eg\mapsto (e,g,g^{-1}(r(e)))$, $ge^*\mapsto (g(r(e)),g,e)$. 
\end{prop}
\begin{proof} One checks that the elements
\[(v,1,v),\ (v,g,g^{-1}(v)),\ (e,1,r(e)),\ (e,g,g^{-1}(r(e))),\ (r(e),1,e)\ \text{and } (g(r(e)),g,e)\]
of $\ell[\cS(G,E,\phi),\omega]$ satisfy the relations \eqref{eq:cohn0}, \eqref{eq:cohn1} and \eqref{eq:cohn2}; thus there is a homomorphism as in the proposition. Next we define its inverse. Because $e=s(e)er(e)$ in $C(G,E,\phi_c)$, we have a multiplication preserving map $\cP(E)\to C(G,E,\phi_c)$ sending a path $\alpha=e_1\cdots e_n$ to the product of its edges in $C(G,E,\phi_c)$. If $\alpha,\beta\in\cP(E)$ and $g\in G$, we write 
\[
\alpha g\beta^*:=\alpha (r(\alpha)g)\beta^*\in C(G,E,\phi_c)
\]
If $\alpha$ or $\beta$ is a vertex, we abbreviate the above as $\alpha g$ or $g\beta^*$.
The set-theoretic map 
\begin{equation}\label{map:sgetocohn}
a:\cS(G,E,\phi)\to C(G,E,\phi_c),\ \ a(0)=0,\ \ a(\alpha,g,\beta)=\alpha g\beta^*
\end{equation} 
induces a homomorphism of $\ell$-modules $\ell[\cS(G,E,\phi),\omega]\to C(G,E,\phi_c)$. We shall show that this is in fact a homomorphism of algebras. The usual proof that the Cohn algebra $C(E)=C(1,E,1)$ is the semigroup algebra of $\cS(E)=\cS(1,E)$ shows that the relations \eqref{eq:cohn0} and
\eqref{eq:cohn1} imply that, if $\beta,\gamma\in \cP(E)$, then
\begin{equation}\label{eq:plaincohn}
\beta^*\gamma=\left\{\begin{matrix}\beta_1^* &\text{ if } \beta=\gamma\beta_1\\
                                   \gamma_1 &\text{ if } \gamma=\beta\gamma_1\\
																	         0 &\text{ otherwise.}\end{matrix}\right.
\end{equation}
Next we show, by induction on $|\gamma|$, that 
\begin{equation}\label{eq:ggamma}
(vg)\gamma=\delta_{v,g(s(\gamma))}g(\gamma)\phi_c(g,\gamma).
\end{equation}
If $|\gamma|\le 1$ this is clear from the second equation of \eqref{eq:cohn1} and the first in \eqref{eq:cohn2}. For the inductive step, let $n\ge 1$, $|\gamma|=n+1$, $\gamma=e_{n+1}\cdots e_1$, $w=s(\gamma)$, $u=r(\gamma)$. Then using the first equation of \eqref{eq:cohn2} in the first line, the inductive step in the second, and \eqref{eq:conditwist} and part iii) of Lemma \ref{lem:epextend} and Remark \ref{rem:phic} in the last, we obtain
\begin{gather*}
(vg)\gamma=\delta_{v,g(w)}g(e_{n+1})\cdot (g(r(e_{n+1}))\phi_c(g,e_{n+1}))\cdot (e_n\cdots e_1)\\
=\delta_{v,g(w)}c(g,e_{n+1})
g(e_{n+1})\phi(g,e_{n+1})(e_n\cdots e_1)
\phi_c(\phi(g,e_{n+1}), e_n\cdots e_1)\\
=\delta_{v,g(w)}g(\gamma)\phi_c(g,\gamma).
\end{gather*}
A similar argument shows that 
\begin{equation}\label{eq:ggamma*}
\gamma^*(vg)=\delta_{v,s(\gamma)}\phi_c(g,g^{-1}(\gamma))g^{-1}(\gamma)^*.
\end{equation}
Next we use \eqref{eq:plaincohn}, \eqref{eq:ggamma} and \eqref{eq:ggamma*}  to show that \eqref{map:sgetocohn} induces an algebra homomorphism $\ell[\cS(G,E,\phi),\omega]\to C(G,E,\phi_c)$. 
If, for example, $(\alpha,g,\beta)$, $(\gamma,h,\theta)\in\cS(G,E,\phi)$ and $\gamma=\beta\gamma_1$, then using \eqref{eq:plaincohn}, \eqref{eq:ggamma} and \eqref{prod:sge} we obtain
\begin{align*}
a(\alpha,g,\beta)a(\gamma,h,\theta)=&(\alpha g\beta^*)(\gamma h\theta)\\
=&\alpha g\cdot\gamma_1 h\theta\\
                                 =&c(g,\gamma_1)\alpha g(\gamma_1)\phi(g,\gamma_1)h\theta^*\\
																 =&\omega((\alpha,g,\beta),(\gamma,h,\theta))a((\alpha,g,\beta)(\gamma,h,\theta))
\end{align*}
The case $\beta=\gamma\beta_1$ is similar, using \eqref{eq:ggamma*} in place of \eqref{eq:ggamma}. 
Thus $a$ is an algebra homomorphism. It is clear that the composite of $a$ with the map of the lemma is the identity on $C(G,E,\phi_c)$. To see that the reverse composition is the identity also, one checks that if $\alpha=e_1\cdots e_n$ and $\beta=f_1\cdots f_m$ and $g\in G$ are such that $g(r(\beta))=r(\alpha)$, then the following identity holds in $\cS(G,E,\phi)$
\[
(\alpha,g,\beta)=(e_1,1,r(e_1))\cdots (e_n,1,r(e_n))(r(e_n),g,r(f_m))(r(f_m),1,f_m)\cdots (r(f_1),1,f_1).
\]
This completes the proof.
\end{proof}
\begin{coro}\label{coro:cohnpres}
$C(G,E,\phi_c)$ is a free $\ell$-module with basis $\cB=\{\alpha g\beta^*:\alpha,\beta\in\cP(E),\ \ g\in G,\ \ g(r(\beta))=r(\alpha)\}$. 
\end{coro}

\begin{rem}\label{rem:ggammas1}
If $E^0$ is finite, then we may add up \eqref{eq:ggamma} over all $v\in E^0$ to obtain
\[
g\gamma=g(\gamma)\phi_c(g,\gamma).
\]
Similarly, from \eqref{eq:ggamma*} we get
\[
\gamma^*g=\phi_c(g,g^{-1}(\gamma))g^{-1}(\gamma)^*.
\]
\end{rem}
\subsection{The algebra \topdf{$\cK(G,E,\phi_c)$}{K(G,E,phic)}}\label{subsec:kgecp}

For each regular vertex $v$ of $E$ and each $g\in G$ consider the following element of $C(G,E,\phi_c)$
\begin{equation}\label{eq:qv}
q_vg:=vg-\sum_{s(e)=v}e\phi_c(g,g^{-1}(e))g^{-1}(e)^*.
\end{equation}
As usual we set $q_v=q_v1$; observe that
\[q_vg=q_v\cdot (vg).\]
Recall that a vertex $v\in E^0$ is \emph{regular} if $0<|s^{-1}\{v\}|<\infty$; write $\reg(E)\subset E^0$ for the subset of regular vertices. Define a two-sided ideal
\begin{equation}\label{eq:kgec}
C(G,E,\phi_c)\vartriangleright \cK(G,E,\phi_c)=\langle q_v:v\in \reg(E)\rangle.
\end{equation}
For each $v\in E^0$ let 
\begin{equation}\label{parribajo}
\cP_v=\{\alpha\in\cP(E):r(\alpha)=v\},\ \ \cP^v=\{\alpha\in\cP(E):s(\alpha)=v\}.
\end{equation}
The action of $G$ on $\cP(E)$ induces one on $\bigoplus_{v\in E^0} M_{\cP_v}$, which maps $\bigoplus_{v\in\reg(E)} M_{\cP_v}$ to itself. In particular we can form the crossed (or smash) product algebra 
\[
(\bigoplus_{v\in\reg(E)} M_{\cP_v})\rtimes G.
\] 
Multiplication in the crossed product is defined by
\begin{equation}\label{eq:crosp1}
(\epsilon_{\alpha,\beta}\rtimes g)(\epsilon_{\gamma,\theta}\rtimes h)=\delta_{\beta,g(\gamma)}\epsilon_{\alpha,g(\theta)}\rtimes gh
\end{equation}

\begin{prop}\label{prop:kge} There is an algebra isomorphism
\[
\upsilon:(\bigoplus_{v\in\reg(E)} M_{\cP_v})\rtimes G\to\cK(G,E,\phi_c),\ \ \epsilon_{\alpha,\beta}\rtimes g \mapsto \alpha (q_{r(\alpha)}g) (g^{-1}(\beta))^*.
\]
In particular, $\cK(G,E,\phi_c)$ is independent of $\phi_c$ up to canonical algebra isomorphism.
\end{prop}

\begin{proof}
One checks that for $v,w\in \reg(E)$ and $g,h\in G$, we have
\begin{equation}\label{eq:qvqw}
(q_vg)(q_wh)=\delta_{v,g(w)}q_v(gh). 
\end{equation}
Thus if $\alpha,\beta\in\cP_v$ and $\gamma,\theta\in\cP_w$, 
\begin{gather*}
\upsilon(\epsilon_{\alpha,\beta}\rtimes g)\upsilon(\epsilon_{\gamma,\theta}\rtimes h)=(\alpha (q_vg) g^{-1}(\beta)^*)(\gamma (q_wh) h^{-1}(\theta)^*)\\
=\delta_{g(\gamma),\beta}\alpha q_v(gh) h^{-1}(\theta)^*=\upsilon(\epsilon_{\alpha,g(\theta)}\rtimes gh).
\end{gather*}
Hence $\upsilon$ is a homomorphism of algebras. To prove that $\upsilon$ is bijective, it suffices to show that 
\begin{equation}\label{B'}
\cB'=\{\alpha q_vg\beta^*:\ \ v\in \reg(E),\ \ g\in G,\ \ \alpha\in\cP_v,\ \ \beta\in \cP_{g^{-1}(v)}\}
\end{equation}
is an $\ell$-module basis of $\cK(G,E,\phi_c)$. By Corollary \ref{coro:cohnpres}, $\cK(G,E,\phi_c)$ is generated by the products $x (q_vg)y$ with $x,y\in\cB$. One checks that if $|\alpha|\ge 1$, then
$\alpha^*(q_vg)=(q_vg)\alpha=0$. It follows that $\cB'$ generates $\cK(G,E,\phi_c)$ as an $\ell$-module. Linear independence of $\cB'$ is derived from that of $\cB$, just as in the case $G=1$ (\cite{lpabook}*{Proposition 1.5.11}; see also \cite{notas}*{Proposici\'on 4.3.3}). 
\end{proof}

\begin{coro}\label{coro:kge} The $*$-algebra $\cK(G,E,\phi_c)$ carries a canonical $G$-grading, where for each $g\in G$, the homogeneous component of degree $g$ is 
\[
\cK(G,E,\phi_c)_g=\mspan\{\alpha (q_vg)\beta^*:v\in \reg(E),\ \ g\in G,\ \ \alpha\in\cP_v,\ \ \beta\in \cP_{g^{-1}(v)}\}.
\]
\end{coro}
\subsection{Extending \topdf{$\cB'$}{B'} to a basis of \topdf{$C(G,E,\phi_c)$}{C(G,E,phic)}}\label{subsec:basextend}
Let $(G,E,\phi_c)$ be a twisted EP-tuple and let $e\in E^1$. Consider the map
\begin{equation}\label{map:nablae}
\nabla_e:G\to E^1\times G,\ \ g\mapsto (g^{-1}(e),\phi(g,g^{-1}(e))).
\end{equation}
In Lemma \ref{lem:ann=0} below we show equivalent conditions to the injectivity of the map $\nabla_e$ of \eqref{map:nablae}. First recall some definitions and notations. 
Let $R$ be a ring, $M$ a right $R$-module and $x\in M$. Write 
$$\ann_R(x)=\{a\in R:xa=0\}$$
for the \emph{annihilator} of $x$. 
Let $(G,E,\phi_c)$ be a twisted EP-tuple. The $\ell$-module $\ell[\cS(G,E,\phi)]$ has two different right $\ell[G]$-module structures, induced respectively by multiplication in $\cS(G,E,\phi)$ and by that in $\ell[\cS(G,E,\phi),\omega]$:
\begin{gather*}
(\alpha,g,\beta)\cdot h=(\alpha,g,\beta)(s(\beta),h,h^{-1}(s(\beta))\text{ and }\\
(\alpha,g,\beta)\cdot_\omega h=(\alpha,g,\beta)\cdot_\omega(s(\beta),h,h^{-1}(s(\beta)).
\end{gather*}
Observe that if $e\in E^1$, then an element $\sum_{g\in G}\lambda_gg\in\ell[G]$ is in the annihilator of $(e,1,e)$ with respect to the first structure if and only if $\sum_{g\in G}\lambda_gc(g,g^{-1}(e))^{-1}g$ is in the annihilator with respect to the second structure. Hence in view of Proposition \ref{prop:cohnpres}, the condition
\begin{equation}\label{eq:ann=0}
\ann_{\ell[G]}(ee^*)=0
\end{equation}
does not depend on whether we regard $ee^*$ as an element of $C(G,E,\phi)$ or of $C(G,E,\phi_c)$. Recall from \cite{ep}*{Section 5} that a finite path $\alpha\in\cP(E)$ is \emph{strongly fixed} by an element $g\in G$ if $g(\alpha)=\alpha$ and $\phi(g,\alpha)=1$.

\begin{lem}\label{lem:ann=0} Let $(G,E,\phi_c)$ be a twisted EP-tuple and $e\in E^1$. Then the following are equivalent.
\item{i)} $\ann_{\ell[G]}(ee^*)=0$. 
\item[ii)] $\{g\in G: g$ fixes $e$ strongly $\}=\{1\}$ 
\item[iii)] The map $\nabla_e$ of \eqref{map:nablae} is injective. 
\end{lem}
\begin{proof} The equivalence between ii) and iii) is \cite{ep}*{Proposition 5.6}. We show that i)$\Rightarrow$ii) and that iii)$\Rightarrow$i). For the purpose of this proof, we shall regard $ee^*$ as an element of $C(G,E,\phi)$. If $g\in G\setminus\{1\}$ fixes $e$ strongly, then $0\ne g-1\in \ann_{\ell[G]}(ee^*)$. Next assume that 
there is a nonzero element $x=\sum_{g\in G}\lambda_gg\in \ann_{\ell[G]}(ee^*)$. Then
\[
0=\sum_{g\in G}\lambda_gg\cdot ee^*=\sum_{\{h\in G, f\in G\cdot e\}}(\sum_{\{g| g(e)=f, \phi(g,e)=h\}}\lambda_g) fhe^*.
\]
By Corollary \ref{coro:cohnpres}, we must have 
\[
\sum_{\{g| g(e)=f,\, \phi(g,e)=h\}}\lambda_g=0\ \ (\forall f\in G\cdot e, h\in G).
\]
Since $x\ne 0$, this implies that there are $g\ne g'\in G$ such that $g(e)=g'(e)$ and $\phi(g,e)=\phi(g',e)$, which precisely means that $\nabla_e$ is not injective.
\end{proof}

 Consider the following subsets of $\reg(E)$
\begin{gather}
\reg(E)_0=\{v\,\colon\, \im(\nabla_e)=\{(e,1)\}\, \forall e\in s^{-1}\{v\}\}\label{eq:strongfix}\\
\reg(E)_1= \{v\,\colon\,(\exists e\in s^{-1}\{v\})\,\nabla_e \text{ is injective }\}\label{eq:pseudofree}
 \end{gather}
Let 
\begin{equation}\label{map:ev}
 \reg(E)\to E^1, v\mapsto e_v 
\end{equation} 
be a section of $s$ such that $\nabla_{e_v}$ is injective for all $v\in\reg(E)_1$.   
Set
\begin{gather*}
\cB\supset \cA=\{\alpha e_v\phi(g,g^{-1}(e_v))(\beta g^{-1}(e_v))^*,\ v\in\reg(E),\  g\in G\}\\
\cup\{\alpha vg\beta^*\,\colon\, v\in\reg(E)_0,\, g\in G\setminus\{1\}\}
    \end{gather*}

\begin{prop}\label{prop:basisunion}
Let $\cB'$ be as in \eqref{B'}. Assume that $\reg(E)=\reg(E)_0\sqcup \reg(E)_1$. Set $\cB"=\cB\setminus\cA$. 
  Then $\cB'\cup\cB"$ is an $\ell$-module basis of $C(G,E,\phi_c)$.
\end{prop}
\begin{proof}
Put $\cB'''=\cB'\cup\cB"$ and let $M$ be the $\ell$-submodule generated by $\cB'''$. Let $v\in E^0$ and $g\in G$, such that $r(\alpha)=v=g(r(\beta))$. If $v\in\reg(E)_0$, then
\[
\alpha vg\beta^*=\alpha q_vg\beta^*-\alpha q_v\beta^*+\alpha\beta^*\in M.
\]
For any $v\in\reg(E)$
\begin{gather*}
\alpha e_v\phi(g,g^{-1}(e_v))(\beta g^{-1}(e_v))^*=\\
c(g^{-1},e_v)\left(\alpha vg\beta^*-\alpha q_vg\beta^*-
\sum_{s(e)=v, e\ne e_v}\alpha e\phi_c(g,g^{-1}(e))(\beta g^{-1}(e))^*\right)\in M.
\end{gather*}
Thus $M=C(G,E,\phi_c)$. Next we show that $\cB'''$ is linearly independent. Let $N$ be the linear span of $\cB"$. By Proposition \ref{prop:kge} and Corollary \ref{coro:cohnpres},
$\cB'$ is a basis of $\cK=\cK(G,E,\phi_c)$ and $\cB"$ is a basis of $N$. Hence it suffices to show that $\cK\cap N=0$. Suppose otherwise that $0\ne x\in \cK\cap N$. Let 
$\alpha q_vg\beta^*$ be an element of $\cB'$ with $d=|\alpha|$ maximum among those appearing with a nonzero coefficient in the unique expression of $x$ as an $\ell$-linear 
combination of $\cB'$. Because $x$ is also a linear combination of $\cB"$, and $\alpha e_ve_v^*g\beta^*=\alpha e_v\phi(g,g^{-1}(e_v))g^{-1}(e_v)^*\beta^*\notin \cB"$, there must be 
another element of $\cB'$ appearing in the expression of $x$ that cancels the latter term. Because $d$ is maximum, that other element must be of the form $\alpha q_vh\beta^*$ for 
some $h\ne g\in G$ such that $h^{-1}(e_v)=g^{-1}(e_v)$  and $\phi(h,h^{-1}(e_v))=\phi(g,g^{-1}(e_v))$. By our hypothesis on $E$, this implies that $v\in\reg(E)_0$. It follows that $\alpha g\beta^*$ and $\alpha h\beta^*$ appear in the unique 
expression of $x$ as a linear combination of $\cB$, and because we are assuming that  $x\in N$, both basis elements must be in $\cB"$, a contradiction. This completes the proof. 
\end{proof}
We say that a twisted EP tuple $(G,E,\phi_c)$ is \emph{pseudo-free} if the associated untwisted EP-tuple $(G,E,\phi)$ is pseudo-free in the sense of \cite{ep}*{Definition 5.4}, which means precisely that the conditions of Lemma \ref{lem:ann=0} hold for every $e\in E^1$. We call $(G,E,\phi_c)$ \emph{partially pseudo-free} if $\reg(E)=\reg(E)_1$. In this case we have
\begin{equation}\label{eq:b''pseudo}
\cB"=\cB\setminus\{\alpha e_v\phi(g,g^{-1}(e_v))(\beta g^{-1}(e_v))^*,\ v\in\reg(E),\  g\in G\}.
\end{equation}

\subsection{The twisted Exel-Pardo algebra \topdf{$L(G,E,\phi_c)$}{L(G,E,phic)}}\label{subsec:lgefdefi}

Let $(G,E,\phi_c)$ be a twisted EP-tuple. Set
\[
L(G,E,\phi_c)=C(G,E,\phi_c)/\cK(G,E,\phi_c).
\]
The following is a corollary of Proposition \ref{prop:basisunion}. 
\begin{coro}\label{coro:basisunion}
Let $(G,E,\phi_c)$ be a twisted EP tuple such that $\reg(E)=\reg(E)_0\sqcup \reg(E)_1$ and let $\cB"\subset C(G,E,\phi_c)$ be as in Proposition \ref{prop:basisunion}. Then the image of $\cB"$ is a basis of $L(G,E,\phi_c)$. In particular the algebra extension
\[
\cK(G,E,\phi_c)\into C(G,E,\phi_c)\onto L(G,E,\phi_c)
\]
is $\ell$-linearly split. 
\end{coro}

\begin{prop}\label{prop:lesubsetlge}
Let $(G,E,\phi_c)$ be any twisted EP tuple. Then the canonical homomorphism $L(E)\to L(G,E,\phi_c)$ is injective. 
\end{prop}
\begin{proof}
Let \eqref{map:ev} be any section of $s$. By \cite{lpabook}*{Corollary 1.5.12} the following subset of $C(E)$
$$\cB_0=\{\alpha\beta^*:r(\alpha)=r(\beta)\}\setminus\{\alpha e_v(\beta e_v)^*:r(\alpha)=s(\alpha)=v\in\reg(E)\}$$
maps to a basis of $L(E)$. Hence it suffices to show that $\cB_0\cup\cB'\subset C(G,E,\phi_c)$ is linearly independent. As in the proof of Proposition \ref{prop:basisunion}, this amounts to showing that
the $\ell$-submodule $C(G,E,\phi_c)\supset N=\cK(G,E,\phi_c)\cap \mspan \cB_0$ is zero. If $0\ne x\in N$ then it is a linear combination of elements of $\cB_0\subset\cB$, so any element $\alpha q_vg\beta^*$ appearing with a nonzero coefficient in the unique expression of $x$ as a linear combination of $\cB'$ must have $g=1$. So $x\in\cK(1,E,1)\cap\mspan (\cB_0)$ which is zero by \cite{lpabook}*{Corollary 1.5.12}.
\end{proof}

\section{Twisted Exel-Pardo algebras as twisted Steinberg algebras}\label{sec:epstei}

\subsection{From semigroup twists to groupoid twists}\label{subsec:twistgrpd}

Let $(\cS,\emptyset)$ be a pointed inverse semigroup, and let $\omega:\cS\times\cS\to \cU(\ell)_\bu$ be a $2$-cocycle as in Subsection \ref{subsec:twistsemi}. Equip the smash product set
\begin{align*}
\cU(\ell)_\bu\land \cS=&\ \ \negthickspace \cU(\ell)_\bu\times \cS/\cU(\ell)_\bu\times\{\emptyset\}\cup \{0\}\times \cS\\
                            =&\ \ \negthickspace \cU(\ell)\times \cS/\cU(\ell)\times\{\emptyset\}.
\end{align*}
with the product
\[
(u\land s)(v\land t)=uv\omega(s,t)\land st.
\]
The result is a pointed inverse semigroup $\tilde{\cS}$, with inverse 
\[
(u\land s)^*=u^{-1}\land s^*.
\]
We write $0$ for the class of the zero element in $\tilde{\cS}$. The coordinate projection $\cU(\ell)\times \cS\to \cS$ induces a surjective semigroup homomorphism $\pi:\tilde{S}\to S$. Observe that $\pi^{-1}(\{\emptyset\})=\{0\}$ and that if $p\in\cS$ is a nonzero idempotent, we have a group isomorphism
\[
\pi^{-1}(\{p\})=\cU(\ell)\land p\cong \cU(\ell).
\]
Next let $X$ be a locally compact, Hausdorff space, and let 
\[
\cI_{\mathrm{cont}}(X)=\{f:U\to X\ | \ U\subset X \text{ open, } f \text{ injective and continuous}\}
\]
be the inverse semigroup of partially defined injective maps with open domains. Recall from \cite{exel}*{Definition 4.3} that an 
\emph{action} of $\cS$ on $X$ is a semigroup homomorphism $\theta:\cS\to \cI_{\mathrm{cont}}(X)$, $s\mapsto\theta_s$, such that 
\[
X=\bigcup_{s\in \cS}\dom(\theta_s).
\]
To alleviate notation, for $s\in S$ we write 
\[
\dom (s)=\dom(\theta_s), \text{ and for } x\in\dom (s),\ \ s(x)=\theta_s(x).
\]
A \emph{germ} for the action of $\cS$ on $X$ is the class $[s,x]$ of a pair $(s,x)\in\cS\times X$ with $x\in\dom(s)$, where $[s,x]=[t,y]$ if and only if $x=y$ and there exists an idempotent $p\in\cS$ such that $p(x)=x$ and $sp=tp$. The set of all germs forms an \'etale groupoid $\cG=\cG(\cS,X)$ \cite{exel}*{Proposition 4.17}, topologized by the basis of open sets
\[
[s,U]=\{[s,x]:x\in U\}, \ \ s\in\cS, \ U\subset X \text{ open such that }\dom(s)\supset U .
\]
A germ $[s,x]$ has domain $x$ and range $s(x)$ and the multiplication map $\cG^{(2)}\to \cG$ is given by
\[
[s,t(x)][t,x]=[st,x].
\]
Assume an action of $\cS$ on $X$ is given; then $\tilde{\cS}$ also acts on $X$  via $\theta\circ\pi$, and so we can consider the groupoids of germs 
\[
\cG=\cG(\cS,X), \ \ \tilde{\cG}=\cG(\tilde{\cS},X).
\]
Equip $\cU(\ell)$ with the discrete topology and regard it as a groupoid over the one point space. 
The map 
\begin{equation}\label{map:trivbund}
\tilde{\cG}\to \cU(\ell)\times \cG, \ \ [u\land s,x]\mapsto (u,[s,x])
\end{equation}
is a homeomorphism mapping a basic open set $[u\land s,U]$ to the basic open $\{u\}\times[s,U]$, and the sequence
\begin{equation}\label{seq:distwist}
\cU(\ell)\times\cG^{0}\to \tilde{\cG}\to \cG
\end{equation}
is a \emph{discrete twist} over the (possibly non-Hausdorff) \'etale groupoid $\cG$, in the sense of \cite{resteintwist}*{Section 2.3}. Under the homeomorphism \eqref{map:trivbund}  multiplication of $\tilde{\cG}$ corresponds to
\[
(u,[s,t(x)])(v,[t,x])=(uv\omega(s,t),[st,x]).
\]
The map
\begin{equation}\label{map:tildomega}
\tilde{\omega}:\cG^{(2)}\to\cU(\ell), \ \ \tilde{\omega}([s,t(x)],[t,x])=\omega(s,t)
\end{equation}
is a continuous, normalized $2$-cocycle in the sense of \cite{steintwist}*{page 4}. 

Recall that a \emph{slice} (or \emph{local bisection}) is an open subset $U\subset\cG$ such that the domain and codomain functions restrict to injections on $U$. For example, the basic open subsets $[s,U]$ are slices \cite{exel}*{Proposition 4.18}. We observe moreover that the domain map induces a homeomorphism $[s,U]\iso U$ \cite{exel}*{Proposition 4.15}. The groupoid $\cG$ is \emph{ample} if its topology has a basis of compact slices. Such is the case, for example, if $X$ has a basis of clopen subsets, in which case we say that $X$ is \emph{Boolean}. If $\cG$ is ample, its \emph{Steinberg} algebra $\cA(\cG)$ \cite{steinappr}*{Proposition 4.3} is the $\ell$-module spanned by all  characteristic functions of compact slices, equipped with the convolution product
\begin{equation}\label{convo}
f\star g([s,x])=\sum_{t_1t_2=s}f[t_1,t_2(x)]g[t_2,x].
\end{equation}

Following \cite{steinappr}*{Definition 5.2}, we say that the action of $\cS$ on $X$ is \emph{ample} if $X$ is Boolean and $\dom(s)$ is a compact clopen subset for all $s\in\cS$. In this case the characteristic function $\chi_{[s,\dom(s)]}\in\cA(\cG)$ for every $s\in\cS$, and the $\ell$-module map
\begin{equation}\label{map:ellstoag}
\ell[\cS]\to \cA(\cG),\ \ s\mapsto \chi_{[s,\dom(s)]}
\end{equation}
is a homomorphism of algebras. One may also equip the $\ell$-module $\cA(\cG)$ with the \emph{twisted convolution product}
\begin{align}\label{convotwist}
f\star_{\omega} g([s,x])=&\sum_{t_1t_2=s}\tilde{\omega}([t_1,t_2(x)], [t_2,x])f[t_1,t_2(x)]g[t_2,x]\\
=&\sum_{t_1t_2=s}\omega(t_1,t_2)f[t_1,t_2(x)]g[t_2,x].\nonumber
\end{align}
The result is an algebra $\cA(\cG,\tilde{\omega})$. By \cite{steinappr}*{Proposition 4.3}, this is the same as the twisted Steinberg algebra defined in \cite{steintwist}*{Proposition 3.2} under the assumption that $\cG$ is Hausdorff. Moreover, by \cite{steintwist}*{Corollary 4.25}, the latter also agrees with the Steinberg algebra of the discrete twist \eqref{seq:distwist}. 
Comparison of the formulas \eqref{convo}, \eqref{convotwist} and \eqref{semitwist} tells us that whenever the $\ell$-linear map \eqref{map:ellstoag} is an algebra homomorphism $\ell[\cS]\to \cA(\cG)$, it is also an algebra homomorphism $\ell[\cS,\omega]\to \cA(\cG,\tilde{\omega})$. 

\begin{lem}\label{lem:ellStostein}
Let $\cS$ be an inverse semigroup, $\omega:\cS\times\cS\to \cU(\ell)$ a $2$-cocycle, $\cS\to \cI_{\mathrm{cont}}(X)$ an ample action and $\cG$ the groupoid of germs. Let $\tilde{\omega}:\cG^{(2)}\to\cU(\ell)$ be as \eqref{map:tildomega}. Assume that the algebra homomorphism \eqref{map:ellstoag} is surjective with kernel $\cK$. Then $\cK\triqui\ell[\cS,\omega]$ is an ideal, and $\ell[\cS,\omega]/\cK\cong \cA(\cG,\tilde{\omega})$.
\end{lem}
\begin{proof} This follows from the discussion immediately above the lemma. 
\end{proof}
\subsection{The twisted Exel-Pardo algebra as a twisted groupoid algebra}\label{subsec:twistgpd}

Let $E$ be a graph. Set
\[
\fX(E)=\{\alpha: \text{ infinite path in } E\}\cup\{\alpha\in\cP(E):r(\alpha)\in\sing(E)\}.
\]
For $\alpha\in\cP(E)$, let
\[
\fX(E)\supset \cZ(\alpha)=\{x\in \fX(E)\ :\ x=\alpha y,\  y\in\fX(E)\}=\alpha\fX(E).
\]
For each finite set $F\subset \alpha\cP(E)$, let 
\[
\cZ(\alpha\setminus F)=\cZ(\alpha)\cap(\bigcup_{\beta\in F}\cZ(\alpha\beta))^c.
\]
The sets $\cZ(\alpha\setminus F)$ are a basis of compact open sets for a locally compact, Hausdorff topology on $\fX(E)$ \cite{web}*{Theorem 2.1}. We regard the latter as a topological space equipped with this topology; it is a Boolean space.  
Now  assume an Exel-Pardo tuple $(G,E,\phi)$ is given. Then there is an action of $\cS(G,E,\phi)$ on $\fX(E)$ such that $\dom(\alpha,g,\beta)=\cZ(\beta)$ and 
\[
(\alpha,g,\beta)(\beta x)=\alpha g(x)\in\cZ(\alpha)
\]
for all $x\in\beta\fX(E)$. Let $\cG(G,E,\phi)$ be the groupoid of germs associated to this action; it is an ample groupoid. If moreover $c:G\times E^1\to\cU(\ell)$ is a $1$-cocycle, and $\omega:\cS(G,E,\phi)\times\cS(G,E,\phi)\to\cU(\ell)_\bu$ is the semigroup $2$-cocycle of \eqref{map:omega}, then by \eqref{map:tildomega} we also have a twist 
\begin{equation}\label{map:tildomega2}
\tilde{\omega}:\cG(G,E,\phi)\times\cG(G,E,\phi)\to \cU(\ell).
\end{equation}
\begin{prop}\label{prop:lgec=stein} Let $(G,E,\phi_c)$ be a twisted EP tuple and let \eqref{map:tildomega2} be the associated groupoid cocycle. Let $L(G,E,\phi_c)$ and $\cA(\cG(G,E,\phi),\tilde{\omega})$ be the twisted Exel-Pardo and Steinberg algebras. Then there is an algebra isomorphism 
\[
L(G,E,\phi_c)\iso \cA(\cG(G,E,\phi),\tilde{\omega}),\ \ \alpha g\beta^*\mapsto\chi_{[(\alpha,g,\beta),\cZ(\beta)]}.
\]
\end{prop}
\begin{proof}
By \cite{cep}*{Theorem 6.4 and the discussion preceding it}, for $\cS=\cS(G,E,\phi)$ the homomorphism \eqref{map:ellstoag} (called $\psi$ in \cite{cep}) descends to an isomorphism $L(G,E,\phi)=\ell[\cS(G,E,\phi)]/\cK(G,E,\phi)\to \cA(G,E,\phi)$. Hence by Lemma \ref{lem:ellStostein} and Proposition \ref{prop:cohnpres} the same map also induces an isomorphism $L(G,E,\phi_c)\to \cA(\cG(G,E,\phi),\tilde{\omega})$. This is precisely the isomorphism of the proposition. 
\end{proof}

\subsection{Simplicity of twisted Exel-Pardo algebras}\label{subsec:simptwist}

In the following proposition and elsewhere, by the \emph{support} of a function $f:X\to \ell$ we understand the set
\[
\supp(f)=\{x\in X:f(x)\ne 0\}. 
\]
\begin{prop}\label{prop:c*simpgpdsimp}
Let $(G,E,\phi)$ be an EP-tuple with $G$ countable and $E$ countable and regular. Assume that the Exel-Pardo $C^*$-algebra $C^*(G,E,\phi)$ is simple. Then 
\item[i)] $\cG(G,E,\phi)$ is effective and minimal.
\item[ii)] If $\ell$ is a subring of $\C$, then the support of every nonzero element of $\cA(\cG(G,E,\phi))\cong L(G,E,\phi)$ has nonempty interior. 
\end{prop}
\begin{proof} Because $G$ is countable and $E$ is countable and regular, we have $C^*(G,E,\phi)=C^*(\cG(G,E,\phi))$, by \cite{eps}*{Theorems 2.5 and 3.2}. Hence because $C^*(G,E,\phi)$ is assumed to be simple, $\cG(G,E,\phi)$ is minimal and effective by \cite{nonhau}*{Theorem 4.10}. This proves i). Again by \cite{nonhau}*{Theorem 4.10} we obtain that the interior of the support of every function $f$ in Connes' algebra $\cC(\cG(G,E,\phi))$ spanned by the locally continuous and compactly supported functions (defined e.g. in \cite{eps}*{Section 4.1}) is nonempty. This implies i), since as we are assuming $\ell\subset\C$, we have
\[
\cA(\cG(G,E,\phi))=\cA_\ell(\cG(G,E,\phi))\subset \cA_\C(\cG(G,E,\phi))\subset \cC(\cG(G,E,\phi)).
\]
\end{proof}

\begin{coro}\label{coro:c*simpgpdsimp}
If $C^*(G,E,\phi)$ is simple and $\ell$ is a field of characteristic zero, then $L(G,E,\phi)$ is simple. 
\end{coro}
\begin{proof}
By Proposition \ref{coro:c*simpgpdsimp} and \cite{steiszak}*{Theorem A} the corollary holds for $\ell=\Q$. This implies that it holds for any field of characteristic zero, by \cite{steiszak}*{Theorem 5.9}.
\end{proof}

Let $\cG$ be an ample \'etale groupoid with domain function $d:\cG^1\to\cG^0$. Recall from \cite{steintwist}*{page 4} that a continuous $2$-cocycle $\nu:\cG^{(2)}\to \cU(\ell)$ is \emph{normalized } if it satisfies
\begin{equation}\label{eq:normal}
\nu(\xi,d(\xi))=\nu(d(\xi^{-1}),\xi)=1 \ \ \forall\xi\in\cG.
\end{equation}
Multiplication in the twisted Steinberg algebra $\cA(\cG,\nu)$ is defined by the twisted convolution product
\begin{equation}\label{eq:twistconvo}
(f\star_\nu g)(\xi)=\sum_{\{(\xi_1,\xi_2)\in \cG^{(2)}| \xi=\xi_1\xi_2\}}\nu(\xi_1,\xi_2)f(\xi_1)g(\xi_2).
\end{equation}
\begin{lem}\label{lem:1Lf}
Let $\cG$ be an ample groupoid such that $\cG^{(0)}$ is Hausdorff and $\nu:\cG^{(2)}\to \cU(\ell)$ a normalized cocycle. 
Let $L\subset\cG^{(0)}$ be a compact open subset and let $f\in\cA(\cG)$. Then 
\[
f\star_{\nu} \chi_L=f\star\chi_L,\ \chi_L\star_\nu f=\chi_L\star f.
\]
\end{lem}
\begin{proof}
Immediate from \eqref{eq:normal} and \eqref{eq:twistconvo}. 
\end{proof}

The following proposition is a twisted version of \cite{nonhau}*{Theorem 3.11 and Corollary 3.12}.
\begin{prop}\label{prop:gpdunique}
Let $\cG$ be a second countable ample groupoid such that $\cG^0$ is Hausdorff, $\ell$ a field, and $\nu:\cG^{(2)}\to\cU(\ell)$ a normalized $2$-cocycle. Assume
\item[i)] $\cG$ is effective.
\item[ii)] For all $0\ne f\in\cA(\cG,\nu)$, $\supp(f)^{\circ}\ne\emptyset$.

\smallskip

\noindent Then for every nonzero ideal  $0\ne I\triqui\cA(\cG,\nu)$ there exists a nonempty compact open subset $L\subset\cG^{(0)}$ such that $\chi_L\in I$. 
\end{prop}
\begin{proof}
Let $\cA=\cA(\cG,\nu)$ and let $0\ne f\in I\triqui \cA$ be an ideal. By the argument in the first paragraph of the proof of \cite{nonhau}*{Theorem 3.11}, there exists a compact open slice $B\subset\cG$ such that $f$ is a nonzero constant on $B$. Let $g=f\star_\nu \chi_{B^{-1}}$, $b\in B$ and $u=bb^{-1}$. Then $g(u)=\nu(b,b^{-1})f(b)\ne 0$. By Lemma \ref{lem:1Lf}, the rest of the argument of the proof of \cite{nonhau}*{Theorem 3.11} now applies verbatim. 
\end{proof}

The following  proposition is a twisted version of one of the directions of the equivalence established in \cite{nonhau}*{Theorem 3.14}.  

\begin{prop}\label{prop:gpdsimpstein}
Let $\cG$, $\nu$ and $\ell$ be as in Proposition \ref{prop:gpdunique}. Further assume that $\cG$ is minimal. Then $\cA(\cG,\nu)$ is simple.
\end{prop}
\begin{proof}
Let $0\ne I\triqui\cA(\cG,\nu)$ be an ideal. By Proposition \ref{prop:gpdunique} there is a nonempty compact open subset $L\subset \cG^0$ such that $\chi_L\in I$.  As in the proof of \cite{nonhau}*{Theorem 3.14} we may choose, for each $x\in \cG^0$, an element $\cG\owns\gamma_x$ with $d(\gamma_x)\in L$ and $d(\gamma^{-1}_x)=x$ and a compact open slice $\gamma_x\in B_x$ . Upon replacing $B_x$ by $B_xL$, we may assume that 
\begin{equation}\label{eq:B=BL}
B_x=B_xL. 
\end{equation}
 Because $\nu$ is locally constant and $B_x$ is a compact open slice, there are $n\ge 1$ and clopen subsets $B_x(1),\dots,B_x(n)\subset B_x$ such that 
$(B_x\times B^{-1}_x)\cap\cG^{(2)}=\bigsqcup_{i=1}^n(B_x(i)\times B^{-1}_x(i))\cap\cG^{(2)}$ and such that $\nu$ is constant on each
$(B_x(i)\times B_x^{-1}(i))\cap\cG^{(2)}$. Let $i$ be such that $\gamma_x\in B_x(i)$. Upon replacing $B_x$ by $B_x(i)$ we may assume that $\nu$ is constant on $(B_x\times B^{-1}_x)\cap \cG^{(2)}$. Let $L_x=B_xLB^{-1}_x$; then $x\in L_x$ and by \eqref{eq:B=BL} and Lemma \ref{lem:1Lf}
\[
\chi_{B_x}\star_\nu\chi_{L}\star_\nu\chi_{B^{-1}_x}=\chi_{B_x}\star_\nu\chi_{B_x^{-1}}=\nu(\gamma_x,\gamma_x^{-1})\chi_{L_x}.
\]
Hence $\chi_{L_x}\in I$, and the proof now follows as in \cite{nonhau}.
\end{proof}

\begin{prop}\label{prop:gpdsimppis}
In the situation of Proposition \ref{prop:gpdsimpstein} assume that for every nonempty compact open subset $L\subset \cG^{(0)}$, the idempotent $\chi_L\in\cA(\cG,\nu)$ is infinite. Then $\cA(\cG,\nu)$ is simple purely infinite. 
\end{prop}
\begin{proof} Let $0\ne f\in\cA(\cG,\nu)$ and let $B\subset\cG$ and $g=f\star_\nu\chi_{B^{-1}}$ be as in the proof of Proposition \ref{prop:gpdunique}. Following the argument of the proof of \cite{nonhau}*{Theorem 3.11} and taking Lemma \ref{lem:1Lf} into account, one obtains a compact open $L\subset \cG^0$ and elements $h_1,h_2\in\cA(\cG,\nu)$ such that 
$$\chi_L=h_1\star_\nu g\star_\nu h_2=h_1\star_\nu f\star_\nu \chi_B^{-1}\star_\nu h_2.$$ 
 
By hypothesis, $\chi_{L}$ is infinite. Hence $\cA(\cG,\nu)$ is simple purely infinite, by \cite{lpabook}*{Proposition 3.1.7 and Definition 3.1.8}.
\end{proof}

\begin{thm}\label{thm:spistein}
Let $(G,E,\phi_c)$ be a twisted EP-tuple with $G$ countable and $E$ countable and regular and let $\ell\supset\Q$ be a field.
If $C^*(G,E,\phi)$ is simple, then $L(G,E,\phi_c)$ is simple. If furthermore, $L(E)$ is simple purely infinite, then so is $L(G,E,\phi_c)$.
\end{thm}
\begin{proof} By Proposition \ref{prop:lgec=stein}, $L(G,E,\phi_c)=\cA(\cG(G,E,\phi),\tilde{\omega})$. Because $G$ and $E$ are countable by hypothesis, $\cG(G,E,\phi)$ is second countable. Hence $L(G,E,\phi_c)$ is simple, by Propositions \ref{prop:c*simpgpdsimp} and \ref{prop:gpdsimpstein} and Corollary \ref{coro:c*simpgpdsimp}. Now suppose that $L(E)$ is simple purely infinite. Consider the groupoid $\cG(E)=\cG(\{1\},E)$; this is the groupoid of \cite{moritagpd}*{Example 2.1}. By \cite{moritagpd}*{Example 3.2}, $L(E)=\cA(\cG(E))$. Since $\cG(E)$ is Hausdorff and since we are assuming that $L(E)$ is simple, $\cG(E)$ is effective and minimal by \cite{steinprisimp}*{Theorem 3.5}, and so, since we are further assuming that $L(E)$ is purely infinite, $\chi_{L}$ is infinite for every compact open subset $L\subset\cG(E)^{(0)}=\cG(G,E,\phi)^{(0)}$. Because of this and because by Proposition \ref{prop:lesubsetlge}, $L(E)$ 
is a subalgebra of $L(G,E,\phi_c)$, $L(G,E,\phi_c)$ is purely infinite simple by Proposition \ref{prop:gpdsimppis}.
\end{proof}

\section{Twisted Katsura algebras}\label{sec:kat}
\numberwithin{equation}{section}
Let $E$ be a row-finite graph, $A=A_E\in \N_0^{(\reg(E)\times E^0)}$ its reduced incidence matrix and $B\in \Z^{(\reg(E)\times E^0)}$ such that 
\begin{equation}\label{katcond0}
A_{v,w}=0\Rightarrow B_{v,w}=0.
\end{equation}
Under the obvious identification $\Z/\Z A_{v,w}\iso \N_{v,w}:=\{0,\dots,A_{v,w}-1\}$, translation by $B_{v,w}$ defines a bijection 
\[
\tau_{v,w}:\N_{v,w}\to \N_{v,w}, \ \ n\mapsto \overline{B_{v,w}+n}
\]
and therefore a $\Z$-action on $\N_{v,w}$. Here $\bar{m}$ is the remainder of $m$ under division by $A_{v,w}$. A $1$-cocyle $\psi_{v,w}:\Z\times \N_{v,w}\to\Z$ for this action is determined by
\[
\psi_{v,w}(1,n)=\frac{B_{v,w}+n-\tau_{v,w}(n)}{A_{v,w}}.
\]
For every pair of vertices $(v,w)$ with $A_{v,w}\ne 0$, choose a bijection 
\[
\n_{v,w}:vE^1w\iso \{0,\dots, A_{v,w}-1\},
\]
and let
\[
\n=\coprod_{v,w}\n_{v,w}:E^1=\coprod_{v,w}vE^1w\to \N_0. 
\]
Consider the $\Z$-action on $E$ that fixes the vertices and is induced by the bijection $\sigma:E^1\to E^1$ that is determined by
\[
\n(\sigma(e))=\tau_{s(e),r(e)}(\n(e)).
\] 
Write $\Z=\langle t\rangle$ multiplicatively and let $\phi:\Z\times E^1\to \Z$ be the $1$-cocycle determined by
\[
\phi(t,e)=t^{\psi_{s(e),r(e)}(1,\n(e))}.
\]
By \cite{hazep}*{Proposition 1.13}, the algebra $L(\Z,E,\phi)$ associated to the EP-tuple $(\Z,E,\phi)$ is the \emph{Katsura algebra} $\cO_{A,B}=\cO_{A,B}(\ell)$, itself the algebraic analogue of the $C^*$-algebra introduced by Katsura in \cite{kat}. Next we consider a twisted version. For this purpose we start with a row-finite matrix
\[
C\in \cU(\ell)^{(\reg(E)\times E^0)}
\]
such that 
\begin{equation}\label{katcond01}
A_{v,w}=0\Rightarrow C_{v,w}=1.
\end{equation}
 
Consider the $1$-cocycle $c:\Z\times E^1\to \cU(\ell)$ defined by
\begin{equation}\label{map:katwist}
c(t,e)=\left\{\begin{matrix}(-1)^{(A_{s(e),r(e)}-1)B_{s(e),r(e)}}C_{s(e),r(e)}& \text{ if } \n(e)=0\\ 1 &\text{ else}\end{matrix}\right.    
\end{equation}

Write $\cO_{A,B}^C=L(\Z,E, \phi_c)$ for the twisted Exel-Pardo algebra associated to the twisted EP-tuple $(\Z,E,\phi_c)$. 

We say that $(A,B)$ is \emph{KSPI} if $E$ is countable, \eqref{katcond0} holds, and if in addition we have
\begin{itemize}
\item  For any pair $(v,w)\in E^0\times E^0$ there exists $\alpha\in\cP(E)$ such that $s(\alpha)=v$ and $r(\alpha)=w$. 
\item  For every vertex $v$ there are at least two distinct loops based at $v$.
\item $B_{v,v}=1$ for all $v\in E^0$. 
\end{itemize}

Katsura proved in \cite{kat}*{Proposition 2.10} that if $(A,B)$ is a Katsura pair, then the $C^*$-algebra completion $C^*_{A,B}=\overline{\cO_{A,B}(\C)}$ is simple and purely 
infinite. 

\begin{thm}\label{thm:katpis}
Let $\ell\supset\Q$ be a subfield. Let $E$ be a countable regular graph, $A=A_E$ its incidence matrix, and $B\in \Z^{(E^0\times E^0)}$ such that $(A,B)$ is a KSPI pair. Let $C\in (\cU(\ell))^{(E^0\times E^0)}$ be arbitrary. Then the twisted Katsura algebra $\cO_{A,B}^C$  is simple purely infinite. 
\end{thm}
\begin{proof} By \cite{kat}*{Proposition 2.10}, $C^*_{A,B}$ is simple. By \cite{lpabook}*{Theorem 3.1.10}, the first two Katsura conditions imply that $L(E)$ simple purely infinite. Hence $\cO_{A,B}^C$ is simple purely infinite by Theorem \ref{thm:spistein}.
\end{proof}

For a general field $\ell$, we have the following more restrictive simplicity criterion.

\begin{prop}\label{prop:katsimp}
Let $E$, $A$ and $B$ be as in Theorem \ref{thm:katpis}, $\ell$ a field and $C\in\cU(\ell)^{(E^0\times E^0)}$ satisfying \eqref{katcond01}. Further assume that

\smallskip

\item[$\bu$] Whenever $A_{v,w}\ne 0=B_{v,w}$, for each $l\ge 1$ there exist finitely many paths $\alpha=e_1\dots e_n\in w\cP(E)$ such that $l\prod_{i=1}^{n-1}B_{s(e_i),r(e_i)}/A_{s(e_i),r(e_i)}\in\Z$.

\smallskip

Then $\cO_{A,B}^C$ is simple.
\end{prop}
\begin{proof}
As mentioned above, Katsura proved that the $C^*$-algebra is simple whenever $(A,B)$ is KSPI. By \cite{ep}*{Theorem 18.6}, the additional condition of the proposition implies that (and is in fact equivalent to) $\cG(G,E,\phi)$ being Hausdorff. Thus by \cite{ep}*{Theorems 18.7, 18.8 and 18.12} $\cG(G,E,\phi)$ is effective (essentially principal in the notation of \cite{ep}) and minimal. Hence $\cO_{A,B}^C=\cA(\cG(G,E,\phi),\tilde{\omega})$ is simple by \cite{steintwist}*{Theorem 6.2}.
\end{proof}

\section{Twisted \topdf{$EP$}{EP}-tuples and \topdf{$kk$}{kk}}\label{sec:kk}
\numberwithin{equation}{subsection}
\subsection{Preliminaries on \topdf{$kk$}{kk}}\label{subsec:prelikk}

Let $\fC$ be a category and $F:\aha\to \fC$ a functor. We say that $F$ is \emph{homotopy invariant} if for every algebra $A$, the map $F(A)\to F(A[t])$ induced by the natural inclusion $A\subset A[t]$ is an isomorphism. Let $X$ be an infinite set and $x\in X$. Define a natural transformation $\iota_x:A\to M_XA$, $\iota_x(a)=\epsilon_{x,x}a$. We say that $F$ is \emph{$M_X$-stable} if the natural transformation $F(\iota_x):F\to F\circ M_X$ is an isomorphism. $M_X$-stability implies $M_Y$-stability for every set $Y$ whose cardinality is at most that of $X$. We fix such an $X$ and use the term \emph{matricially stable} for $M_X$-stable. An algebra extension is a sequence of algebra homomorphisms
\begin{equation}\label{exte}
\cE:\, 0\to A\overset{i}{\lra} B\overset{\pi}{\lra} C\to 0    
\end{equation}
which is exact as a sequence of $\ell$-modules. An \emph{excisive} functor with values in a triangulated category $\cT$ with suspension $T\mapsto T[-1]$ is a functor $F:\aha\to\cT$ together with a family of maps $\partial^F_{\cE}:F(C)[+1]\to F(A)$ indexed by the algebra extensions \eqref{exte}, such that 
\[
\xymatrix{F(C)[+1]\ar[r]^(.6){\partial^F_{\cE}}&F(A)\ar[r]^{F(i)}&F(B)\ar[r]^{F(\pi)}&F(C)}
\]
is a distinguished triangle. The maps $\partial^F_{\cE}$ are to satisfy the compatiblity conditions of \cite{ct}*{Section 6.6}. An excisive, homotopy invariant and matricially stable \emph{homology theory} of $\ell$-algebras consists of a triangulated category $\cT$ and a functor $F:\aha\to\cT$ that is excisive, homotopy invariant and matricially stable. Such homology theories form a category, where a \emph{homomorphism} from $F$ to another homology theory $G:\aha\to\cT'$  is a triangulated functor
$H:\cT\to \cT'$, together with a natural isomorphism $\phi:H(-[+1])\to H(-)[+1] $ such that the following diagram commutes for every extension \eqref{exte}
\begin{equation*}
 \xymatrix{H(F(C)[+1])\ar[d]_{\phi}\ar[r]^(.6){H(\partial^F_{\cE})} & G(A)\\
           G(C)[+1]\ar[ur]^{\partial^G_{\cE}}&}
\end{equation*}
It was shown in \cite{ct}*{Theorem 6.6.2} that the category of homotopy invariant, matricially stable and excisive homology theories of $\ell$-algebras has an initial object $j:\aha\to kk$. We shall also to consider homology theories for the category $G-\aha$ of algebras equipped with a $G$-action and $G$-equivariant homomorphisms and $G_{\gr}-\aha$ of $G$-graded algebras and homogeneous homomorphisms. A \emph{$G$-set} is a set $Y$ with a $G$-action $G\times Y\to Y$, $(g,y)\mapsto g(y)$. A \emph{$G$-graded set} is a set $Y$ together with a function $d:Y\to G$. If $A\in G-\aha$ and $Y$ a $G$-set, $M_YA$ is a $G$-algebra with $g(\epsilon_{y_1,y_2}a)=\epsilon_{g(y_1),g(y_2)}g(a)$. If instead $A\in G_{\gr}-\aha$ and $Y$ is equipped with a $d:Y\to G$, then $M_YA$ is $G$-graded with degree function determined by $|\epsilon_{x,y}a|=d(x)|a|d(y)^{-1}$. Let $X$ be the infinite set fixed above in the definition of matricial stability and $\cT$ a triangulated category. A functor $F$ defined on $G-\aha$  is \emph{$G$-stable} if 
for any $G$-algebra $A$ and any two $G$ sets $Y$, $Z$ of cardinality at most that $G\times X$, the map $F(M_YA\to M_{Y\sqcup Z}A)$ induced by the obvious inclusion is an isomorphism. A functor defined on $G_{\gr}-\aha$ is $G$-stable if it satisfies the same condition for $G$-graded sets $Y$ and $Z$ with the same cardinality restrictions. The universal initial homotopy invariant, matricially stable, $G$-stable and excisive homology theories $j_G:G-\aha\to kk_G$ and $j_{G_{\gr}}:G_{\gr}-\aha\to kk_{G_{\gr}}$ where constructed in \cite{kkg}. It was shown there \cite{kkg}*{Theorem 7.6} that the crossed product functors
$\rtimes G:G=\aha\to G_{\gr}-\aha$ and  $G\hat{\ltimes}:G_{\gr}-\aha\to G-\aha$ induce inverse triangulated equivalences $kk_G\leftrightarrows kk_{G_{\gr}}$.

\subsection{The algebra \topdf{$\cK(G,E,\phi_c)$}{K(G,E,phic)}}\label{sec:kge}

\begin{lem}\label{lem:mxlx}
Let $X$ be a $G$-set. Then the $G$-equivariant homomorphism
\[
\iota_X:\ell^{(X)}\to M_X\ell^{(X)},\ \ \chi_x\mapsto \epsilon_{x,x}\otimes \chi_x 
\]
is a $kk_G$-equivalence. 
\end{lem}
\begin{proof} Let $\{\bu\}$ be a one point $G$-set; consider $X_\bu=X\sqcup\{\bu\}$. By $G$-stability the inclusion $\inc:M_X\ell^{(X)}\to M_{X_\bu}\ell^{(X)}$ is a $kk_G$-equivalence. Hence it suffices to show that $\inc\circ\iota_X$ is one. For each $x\in X$ let $\sigma_x\in M_{X_\bu}$ the matrix associated to the permutation that exchanges $x$ and $\bu$ and fixes the remaining elements. Observe that 
\[
R:=M_{X_\bu}\ell^{(X)}=\bigoplus_{x\in X}M_{X_\bu}\otimes\chi_x\triqui \prod_{x\in X}\Gamma_{X_\bu}\otimes\chi_x=:S.
\]
Let $\sigma=\prod_{x\in X}\sigma_x\otimes\chi_x\in S$. One checks that $\sigma$ is fixed by $G$, hence the conjugation map
$\ad(\sigma):R\to R$ is the identity in $kk_G$, by the argument of \cite{friendly}*{Proposition 2.2.6}. Moreover $\ad(\sigma)\circ\inc\circ\iota_X=\epsilon_{\bu,\bu}\otimes\id_{\ell^{(X)}}$, another $kk_G$-equivalence. Thus $j_G(\inc\circ\iota_X)$ is an isomorphism, finishing the proof. 
\end{proof}

Let $X\subset E^0$ be a set of vertices closed under the $G$-action. Consider the $G$-equivariant homomorphism
\begin{equation}\label{map:eliota}
\iota:\ell^{(X)}\to \bigoplus_{v\in X}M_{\cP_v}\otimes\chi_v,\ \ \iota(\chi_v)=\epsilon_{v,v}\otimes\chi_v.
\end{equation}
\begin{prop}\label{prop:stable} The map \eqref{map:eliota} is a $kk_G$-equivalence.
\end{prop}
\begin{proof} We adapt the argument of the proof of \cite{arcor2}*{Lemma 10.3} as follows. Let $S=\bigoplus_{v\in X}M_{\cP_v}\otimes\chi_v$; set $p=\sum_{\alpha\in\cP_X}\epsilon_{\alpha,\alpha}\otimes\epsilon_{\alpha,\alpha}\otimes\chi_{r(\alpha)}\in\Gamma_{\cP_X}S$. Then 
for $T:=pM_{\cP_X}Sp$, the map
\begin{equation}\label{map:stopmsp}
S\to T,\ \ \epsilon_{\alpha,\beta}\otimes\chi_{r(\alpha)}\mapsto \epsilon_{\alpha,\beta}\otimes\epsilon_{\alpha,\beta}\otimes\chi_{r(\alpha)}
\end{equation}
is a $G$-equivariant isomorphism. Let 
$$u=\sum_{\alpha\in\cP_X}\epsilon_{\alpha,\alpha}\otimes\epsilon_{r(\alpha),\alpha}\otimes\chi_{r(\alpha)}, \ u^*=\sum_{\alpha\in\cP_X}\epsilon_{\alpha,\alpha}\otimes\epsilon_{\alpha,r(\alpha)}\otimes\chi_{r(\alpha)}\in\Gamma_{\cP_X}S.$$ 
Observe that $u$ and $u^*$ are fixed by $G$ and that $u^*u=p$. Put $R=\bigoplus_{v\in  X} \epsilon_{v,v}\otimes \ell\chi_v\subset S$; then $\ad(u):T\to M_{\cP_X}R$, $t\mapsto utu^*$ is a $G$-equivariant homomorphism; by the $G$-equivariant version of \cite{classinvo}*{Lemma 8.11}, the composite of $\ad(u)$ with the inclusion $M_{\cP_X}R\subset T$ equals the identity in $kk_G$. Hence the composite $\zeta$ of $\ad(u)$ with \eqref{map:stopmsp} is a split monomorphism in $kk_G$. Observe that \eqref{map:eliota} is the composite of the isomorphism $f:\ell^{(X)}\to R$, $f(\chi_v)=\epsilon_{v,v}\otimes\chi_v$, with the inclusion $R\subset S$. Consider the composite $\jmath:\ell^{(X)}\to M_{\cP_X}\ell^{(X)}$ of $\iota_X:\ell^{(X)}\to M_X\ell^{(X)}$ with the inclusion $M_X\ell^{(X)}\subset M_{\cP_X}\ell^{(X)}$. The latter maps are $kk_G$-equivalences, by Lemma \ref{lem:mxlx} and $G$-stability, so $\jmath$ is one too. One checks that $\zeta{|R}\circ f=M_{\cP_X}f\circ \jmath$. Hence $\zeta$ is an isomorphism in $kk_G$ and therefore $\iota$ is one also. 
\end{proof}
\begin{prop}\label{prop:kkkge} Let $(G,E,\phi_c)$ be a twisted EP-tuple. Then the homogeneous homomorphism of $G$-graded algebras
\[
q^G:(\ell^{(\reg(E))})\rtimes G\to \cK(G,E,\phi_c),\ \ q^G(\chi_v\rtimes g)=q_vg.
\]
is a $kk_{G_{\gr}}$-equivalence.
\end{prop}
\begin{proof} Let $X=\reg(E)$, $\iota$ as in \eqref{map:eliota} and $\iota'$ the composite of $\iota$ with the obvious isomorphism $\bigoplus_{v\in X}M_{\cP_v}\otimes\chi_v\cong \bigoplus_{v\in X}M_{\cP_v}$. One checks that $q^G$ is the composite of the maps $\iota'\rtimes G$ and $\upsilon$ of Proposition \ref{prop:kge}. Then $j_G(\iota')$ is an isomorphism by Proposition \ref{prop:stable}, whence $j_{G_{\gr}}(\iota'\rtimes G)$ is an isomorphism. This implies that $j_{G_{\gr}}(q^G)$ is an isomorphism, since $\upsilon$ is a $G$-graded algebra isomorphism. 
\end{proof}

\subsection{The Cohn algebra \topdf{$C(G,E,\phi_c)$}{C(G,E,phic)}}

The purpose of this subsection is to prove the following.

\begin{thm}\label{thm:cohnkk}
Let $(G,E,\phi_c)$ be a twisted EP-tuple. Then the algebra homomorphism
\[
\varphi:\ell^{(E^0)}\rtimes G\to C(G,E,\phi_c),\ \ \chi_v\rtimes g\mapsto vg
\]
is a $kk$-equivalence. 
\end{thm}

The proof of Theorem \ref{thm:cohnkk} will be given at the end of the subsection. 
We remark that, in the case when $G$ is trivial, Theorem \ref{thm:cohnkk} specializes to  \cite{cm1}*{Theorem 4.2}. We will adapt the argument of \cite{cm1}*{Theorem 4.2} to prove Theorem \ref{thm:cohnkk}.   

We abbreviate $C(G,E)=C(G,E,\phi_c)$ throughout. Let $\sink(E)\,\inf(E)\subset E^0$ be the subsets of sinks and of infinite emitters. For $v\in E^0\setminus\inf(E)$, consider the following element of $C(G,E)$
\begin{equation}\label{eq:mv}
    m_v=\left\{\begin{matrix} \sum_{s(e)=v}ee^* & \text{ if }v\in\reg(E)\\
                                            0 & \text{ if } v\in\sink(E)
        \end{matrix}\right.
\end{equation}
We define $C^m(G,E)$ as the result of adding an indeterminate $m_v$ for each $v\in\inf(E)$, subject to the relations \cite{cm1}*{Identities 4.4} plus the additional relation 
\begin{equation}\label{eq:mvg}
(g(v)g)m_v=m_{g(v)}g(v)g.
\end{equation}
We remark that \eqref{eq:mvg} is already satisfied in $C(G,E)$,  by all $v\in E^0\setminus\inf(E)$. Thus  in $C^m(G,E)$, \eqref{eq:mvg} holds for all $v\in E^0$. We shall abbreviate $m_vg=m_v(vg)$. The elements $q_v=v-m_v\in C^m(G,E)$ generate an ideal 
\[
\hat{\cK}(G,E)=\cK(G,E)\oplus \bigoplus_{v\in \inf(E)}\mspan\{\alpha q_vg\beta^*: \ \ g\in G,\ \ \alpha\in\cP_v,\ \ \beta\in \cP_{g^{-1}(v)}\}.
\]
One checks that the isomorphism $\upsilon$ of Proposition \ref{prop:kge} extends to an isomorphism 
\[
\hat{\upsilon}:(\bigoplus_{v\in E^0}\cM_{\cP_v})\rtimes G\iso \hat{\cK}(G,E)
\] 
defined by the same formula as $\upsilon$. We have an algebra homomorphism
\begin{equation}\label{map:hatvarphi}
 \hat{\varphi}:\hat{\cK}(G,E)\to M_{\cP}C(G,E), \ \ \hat{\varphi}(\alpha q_vg\beta^*)=\epsilon_{\alpha,\beta}\otimes v g
\end{equation}

Next we consider an analogue of the homomorphism $\rho$ of \cite{cm1}*{(4.7)}. The map therein goes from $C^m(E)$ to $\Gamma_\cP$; the analogue will be a homomorphism
\begin{equation}\label{map:rho}
\rho:C^m(G,E)\to \Gamma^a_{\cP}.
\end{equation}
We define $\rho(e)$, $\rho(e^*)$ for $e\in E^1$ and $\rho(m_v)$ for $v\in\inf(E)$ exactly as in \cite{cm1}*{Formulas (4.7)} and for $v\in E^0$ and $g\in G$ we set 
\begin{equation}\label{eq:rhovg}
\rho(vg)=\sum_{g(r(\alpha))=v}c(g,\alpha)\epsilon_{g(\alpha),\alpha}.
\end{equation}
One checks that the above prescriptions actually define an algebra homomorphism $\rho:C^m(G,E)\to \Gamma^a_{\cP}$. The latter agrees with that of \cite{cm1} when $G$ is trivial, and is therefore injective in that case, by \cite{cm1}*{Lemma 4.8}. For nontrivial $G$, $\rho$ need not be injective (e.g. if $G$ acts trivially on $E$). 

Next we need an analogue of the algebra $\fA$ of \cite{cm1}*{page 131}, and an action of $C^m(G,E)$ on $\fA$; both are defined in Lemma \ref{lem:fa} below. First we introduce some notation and vocabulary. 
Recall (e.g. from \cite{hig}*{Definition 1.1.1}) that the \emph{multiplier algebra} of an algebra $B$ is the set $\cM(B)$ of all pairs $(L,R)$ of $\ell$-linear maps $L,R:B\to B$ such that $L$ is a right $B$-module homomorphism, $R$ is a left $B$-module homomorphism, and 
\[
R(x)y=xL(y) \ \forall x,y\in B.
\]
Following \cite{hig}, elements of $\cM(B)$ are called \emph{double centralizers}. Addition in $\cM(B)$ is defined pointwise and muliplication is given by composition of functions, as follows 
\[
(L_1,R_1)(L_2, R_2)=(L_1\circ L_2, R_2\circ R_1).
\]
Let $A$ be another algebra. An \emph{$A$-algebra} structure on $B$ is an algebra homomorphism $A\to\cM(B)$, $a\mapsto (L_a,R_a)$. We write 
\[
a\cdot b=L_a(b), \ b\cdot a=R_a(b) \ (\forall a\in A, b\in B). 
\]
Thus we have, for all $a\in A$ and $b,c\in B$
\begin{gather}
(a\cdot b)c=a\cdot(bc),\ (b\cdot a)c=b(a\cdot c),\text{ and } b(c\cdot a)=(bc)\cdot a. 
\end{gather}

The \emph{semi-direct product} $A\ltimes B$ is the algebra that results from equipping the direct sum $A\oplus B$ with the following product
\[
(a,b)(a',b')=(aa',a\cdot b'+b\cdot a'+bb').
\]
\begin{lem}\label{lem:fa}
Consider the following $\ell$-submodule of
$M_{\cP}C(G,E)$
\begin{equation}\label{eq:fa}
\fA:=\mspan\{\epsilon_{\gamma,\theta}\otimes\alpha g\beta^*: \ \ s(\alpha)=r(\gamma), s(\beta)=r(\theta), r(\alpha)=g(r(\beta))\}.
\end{equation}
We have the following.
\item[i)] $\fA\subset M_{\cP}C(G,E)$ is a subalgebra.
\item[ii)] $\fA$ carries a $C^m(G,E)$-algebra structure such that for all $x\in C^m(E)$, $g\in G$, $v\in E^0$ and $\epsilon_{\gamma,\theta}\otimes z\in \fA$,
\begin{gather}
x\cdot (\epsilon_{\gamma,\theta}\otimes y)=(\rho(x)\epsilon_{\gamma,\theta})\otimes y,\ (\epsilon_{\gamma,\theta}\otimes y)\cdot x =(\epsilon_{\gamma,\theta})\rho(x)\otimes y\label{eq:lxrx} \\
 vg\cdot (\epsilon_{\gamma,\theta}\otimes y)=\delta_{v,g(s(\gamma))}\epsilon_{g(\gamma),\theta}\otimes (g(r(\gamma))\phi_c(g,\gamma))y,\label{eq:lvg}\\
 (\epsilon_{\gamma,\theta}\otimes y)\cdot vg=\delta_{v,s(\theta)}\epsilon_{\gamma,g^{-1}(\theta)}\otimes y(r(\theta)\phi_c(g,g^{-1}(\theta)).\label{eq:rvg}
\end{gather}
\item[iii)] Let $a\in \hat{K}(G,E)$ and $\fa\in\fA$. For $\hat{\varphi}$ as in \eqref{map:hatvarphi}, we have
\[
a\cdot \fa=\hat{\varphi}(a)\fa, \ \fa\cdot a=\fa\hat{\varphi}(a).
\]
\end{lem}
\begin{proof} Part i) follows from Proposition \ref{prop:cohnpres} using the product formula \eqref{prod:sge} and the definition of $\fA$. Next observe that because $\rho$ is an algebra homomorphism, \eqref{eq:lxrx} defines a homomorphism $C^m(E)\to \cM(\fA)$, $x\mapsto (L_x,R_x)$ where $L_x$ and $R_x$ are left and right multiplication by $\phi(x)\otimes 1$. To prove ii) one shows that this homomorphism extends to $C^m(G,E)$, by sending $vg\in E^0G$ to the pair $(L_{vg},R_{vg})$ of left and right multiplication operators defined by \eqref{eq:lvg} and \eqref{eq:rvg}. This entails first checking that for each $vg\in E^0G$, $(L_{vg},R_{vg})$ is a double centralizer, and then that the relations \eqref{eq:cohn1} and \eqref{eq:cohn2} involving $vg$ are satisfied by $(L_{vg},R_{vg})$ for $vg\in E^0G$. These verifications are long and tedious but straightforward. To prove part iii), observe that if 
$\alpha,\beta$ are paths, and $v\in E^0$ and $g\in G$ are such that $r(\alpha)=v=g(r(\beta))$, then
\begin{align*}
\hat{\phi}(\alpha q_vg\beta^*)=&\epsilon_{\alpha,\beta}\otimes vg\\
=&(\phi(\alpha)\otimes 1)(\epsilon_{v,v}\otimes vg)(\phi(\beta^*)\otimes 1).
\end{align*}
Hence in view of \eqref{eq:lxrx} it suffices to prove iii) for $a=q_vg$ and $\fa=\epsilon_{\gamma,\theta}\otimes\alpha h\beta^*$. This again is a straightforward calculation.  
\end{proof}

\noindent{\em Proof of Theorem \ref{thm:cohnkk}.} We indicate how to modify the proof of \cite{cm1}*{Theorem 4.2} to the present setting. By the argument of Proposition \ref{prop:kkkge}, the algebra homomorphism 
$\hat{\iota}:\ell^{(E^0)}\rtimes G\to \hat{\cK}(G,E)$ is a $kk_{G_{\gr}}$-equivalence, and thus also a $kk$-equivalence. Let
\[
\xi:C(G,E)\to C^m(G,E),\ \ \xi(vg)=m_vg,\ \xi(e)=e m_{r(e)},\ \xi(e^*)=m_{r(e)}e^*.
\]
The map $\xi$ together with the canonical map $\can:C(G,E)\to C^m(G,E)$ form a quasi-homomorphism 
\[
(\can,\xi):C(G,E)\rightrightarrows C^m(G,E)\vartriangleright \hat{\cK}(G,E)
\]
and the argument of \cite{cm1}*{part I of the proof of Theorem 4.2} shows that $j(\hat{\iota})^{-1}\circ j(\can,\xi)$ is left inverse to $\varphi$. Next define
\begin{gather*}
\hat{\iota}_\tau:C(G,E)\to M_{\cP}C(G,E),\ \hat{\iota}_\tau(\alpha g\beta^*)=\epsilon_{s(\alpha),s(\beta)}\otimes \alpha g\beta^*,
\end{gather*}
Remark that 
$$
\hat{\varphi}(q_vg)=\epsilon_{v,g^{-1}(v)}\otimes v g=\hat{\iota}_\tau(vg).
$$
It follows that the analogue of \cite{cm1}*{Diagram 4.16} commutes. Next, we show that $\hat{\iota}_\tau$ is a $kk$-equivalence. As in the proof of Proposition \ref{prop:stable}, we consider the $G$-set with one added fixed point $\cP_\bu=\cP\sqcup\{\bu\}$. Since the inclusion $\inc:M_{\cP}C(G,E)\subset M_{\cP_\bu} C(G,E)$ is a $kk$-equivalence, it suffices to show that the composite $\inc\circ\hat{\iota}_\tau$ is one. For this purpose we consider, for each $v\in E^0$, the matrices $A_v,B_v\in M_{\cP_\bu} C(G,E)[t]$ defined just as in the proof of \cite{cm1}*{Lemma 4.17}, but with $\bu$ substituted for what is called $w$ in the cited reference. One checks that the following prescriptions define a homotopy $H:C(G,E)\to M_{\cP_\bu} C(G,E)[t]$ between $\inc\circ\hat{\iota}_\tau$ and $\iota_\bu$
\begin{gather*}
H(vg)=A_v(\epsilon_{v,g^{-1}(v)}\otimes vg)B_{g^{-1}(v)},\ \ H(e)=A_{s(e)}(\epsilon_{s(e),r(e)}\otimes e)B_{r(e)},\\
H(e^*)=A_{r(e)}(\epsilon_{r(e),s(e)}\otimes e^*)B_{s(e)}.    
\end{gather*}

We have thus come to \cite{cm1}*{Proof of Theorem 4.2, part II}. By Lemma \ref{lem:fa}, we can form the semi-direct product $C^m(G,E)\ltimes \fA$ and consider the $\ell$-submodule 
\[
C^m(G,E)\ltimes \fA\supset J:=\{(x,-\hat{\varphi}(x)): x\in\hat{\cK}(G,E)\}.
\]
Using the fact that $\hat{\varphi}$ is a $*$-homomorphism and part iii) of Lemma \ref{lem:fa}, we obtain that $J$ is a two-sided ideal. Write $D=C^m(G,E)/J$; it is clear that the canonical homomorphisms $\fA\to D\leftarrow C^m(G,E)$ are injective; write $\Upsilon:C^m(G,E)\to D$ for the latter map. We may thus regard $\fA$ as an ideal and $C^m(G,E)$ as a subalgebra of $D$. We can now proceed on to \cite{cm1}*{Proof of Theorem 4.2, part III}. We define homomorphisms $\psi_0,\psi_{1/2},\psi_1:C(G,E)\to D$ just as in the cited reference, with $\xi$, $\can$, $\hat{\iota}_\tau$ and $\Upsilon$ as defined above, so that we have quasi-homomorphisms $(\psi_0,\psi_1),(\psi_0,\psi_{1/2}),(\psi_{1/2},\psi_1):C(G,E)\rightrightarrows D\vartriangleright\fA$. Next, we need an analogue of  \cite{cm1}*{Lemma 4.21}, proving that $j(\psi_0,\psi_{1/2})=0$. One checks that the following prescriptions define an algebra homomorphism  $H^+:C(G,E)\to D[t]$:
\begin{gather*}
H^+(vg)=(m_v g, \epsilon_{v,g^{-1}(v)}\otimes vg), \\ 
H^+(e)=(em_{r(e)},(1-t^2)\epsilon_{s(e),r(e)}\otimes e+t\epsilon_{e,r(e)}\otimes r(e)),\\ H^+(e^*)=(m_{r(e)}e^*,(1-t^2)\epsilon_{r(e),s(e)}\otimes e^*+(2t-t^3)\epsilon_{r(e),e}\otimes r(e)),
\end{gather*}
and that $(H,\psi_{1/2}):C(G,E)\rightrightarrows D[t]\vartriangleright \fA[t]$ is a homotopy $$(\psi_0,\psi_{1/2})\to (\psi_{1/2},\psi_{1/2}).$$ The rest of the proof now proceeds just as in \cite{cm1}.
\qed

\bigskip

\section{Twisted Katsura algebras in \topdf{$kk$}{kk}}\label{sec:kkat}
\numberwithin{equation}{section}

 Let $E$ be a row-finite graph, $A=A_E$ and $\n:E^1\to \N_0$ as in Section \ref{sec:kat}.  Let $v,w\in E^0$ such that $A_{v,w}\ne 0$, and set
\[
m_{v,w}=\sum_{e\in vE^1w}ee^*\in C(\Z,E,\phi_c).
\]
For $0\le i\le A_{v,w}-1$, let $e_i=\n_{v,w}^{-1}(i)\in vE^1w$. Consider the element
\[
u_{v,w}=e_0te^*_{A_{v,w}-1}+\sum_{i=0}^{A_{v,w}-2}e_{i+1}e_i^*\in C(\Z,E,\phi_c).
\]
\begin{lem}\label{lem:uvw}
\item[i)] 
\[
u_{v,w}^{B_{v,w}}=(-1)^{(A_{v,w}-1)B_{v,w}}C_{v,w}^{-1}t\cdot(e_0e_0^*)+\sum_{i=1}^{A_{v,w}-1}t\cdot (e_ie_i^*)
\]
\item[ii)] Let $F\subset E$ be a finite complete subgraph containing $\{v\}\cup r(s^{-1}(\{v\}))$; set $$1_F=\sum_{w\in F^0}w\in C(\Z,E,\phi_c).$$ 
Then the following identity holds in $K_1(C(\Z,F,\phi_c))$
\[
[1_F-m_{v,w}+u_{v,w}]=[1_F-w+(-1)^{(A_{v,w}-1)}wt].
\]
\end{lem}
\begin{proof}

For $m\in\Z$, let $q(m)$ and $\bar{m}$ be the quotient and the remainder under division by $A_{v,w}$. One checks that for any $n\in\N$
\begin{equation}\label{eq:uvwn}
u_{v,w}^n=\sum_{i=0}^{A_{v,w}-1}e_{\ol{i+n}}t^{q(n+i)}e_i.
\end{equation}
Use \eqref{eq:uvwn} at the first step and \eqref{map:katwist} at the second to obtain, for $\psi=\psi_{w,v}$,
\[
u_{v,w}^{B_{v,w}}=\sum_{e\in vE^1w}t(e)t^{\psi(t,e)}e^*=(-1)^{(A_{v,w}-1)B_{v,w}}C_{v,w}^{-1}t\cdot(e_0e_0^*)+\sum_{i=1}^{A_{v,w}-1}t\cdot (e_ie_i^*).
\]
This proves i). Next observe that left multiplication by $u_{v,w}$ defines an automorphism of the projective module $\bigoplus_{i=0}^{A_{v,w}-1}e_ie_i^*C(\Z,F,\phi_c)$. The latter is isomorphic to $P=wC(\Z,F,\phi_c)^{A_{v,w}}$ under the isomorphism defined as left multiplication by $e_i^*$ on the $i$-th summand. Under this isomorphism, $u_{v,w}$ corresponds to the automorphism of $P$ represented by the following matrix with coefficients in $R=wC(\Z,F,\phi_c)w$
\[
\begin{bmatrix} 0&0&\cdots &0&tw\\
                w&0&\cdots &0&0\\
								0&w&\cdots &0&0\\
								\vdots&\vdots& &\vdots&\vdots\\
								0&\cdots&\cdots&w&0
\end{bmatrix}=\begin{bmatrix}tw&0&0&\dots&0\\
                             0& w&0&\dots&0\\
														 0&0&w&\dots&0\\
														 \vdots&\vdots&&\ddots&\vdots\\
														 0&0&0&\cdots &w\end{bmatrix}\cdot \begin{bmatrix} 0&0&\cdots &0&w\\
                w&0&\cdots &0&0\\
								0&w&\cdots &0&0\\
								\vdots&\vdots&\ddots &\vdots&\vdots\\
								0&\cdots&\cdots&w&0\end{bmatrix}
\]

The second matrix is a cyclic permutation matrix with determinant $(-1)^{A_{v,w}-1}w$. Hence the product above is mapped to the class of $(-1)^{A_{v,w}-1}wt$ in $K_1(R)$ and to that of $1-w+(-1)^{A_{v,w}-1}wt$ in $K_1(C(\Z,F,\phi_c))$. Assertion ii) of the lemma is immediate from this. 
\end{proof}
Let $E$ be a graph. We say that a homology theory $H:\aha\to\cT$ is \emph{$E$-stable} if it is $M_X$-stable with respect to a set $X$ of cardinality $\#(E^0\coprod E^1\coprod\N)$. Let $I$ be a set. We say that $H$ is \emph{$E^0$-additive} if first of all direct sums of cardinality $\le \#I$ exist in $\cT$ and second of all the map
\[
\bigoplus_{j\in J}H(A_j)\to H(\bigoplus_{j\in J}A_j)
\]
is an isomorphism for any family of algebras $\{A_j:j\in J\}\subset\aha$ with $\#J\le\# I$. 
\begin{thm}\label{thm:katkkh}
Let $E$ be a row finite graph with reduced incidence matrix $A$,  $B\in \Z^{\reg(E)\times E^0}$ and $C\in \cU(\ell)^{(\reg(E)\times E^0)}$ such that $A_{v,w}=0\Rightarrow B_{v,w}=0$, $C_{v,w}=1$. Let $\cT$ be a triangulated category and $H:\aha\to\cT$ an excisive, homotopy invariant, $E$-stable and $E^0$ additive functor. Consider the matrix
\[
D=\begin{bmatrix} A & 0\\ C & B\end{bmatrix}
\]
Set $C^*\in\cU(\ell)^{E^0\times \reg(E)}$, $C^*_{v,w}=C_{w,v}^{-1}$. Put 
\[
D^*=\begin{bmatrix} A^t & C^*\\ 0 & B^t\end{bmatrix}
\]
Then the Cohn extension of \eqref{intro:cohnext} induces the following distinguished triangle in $\cT$
\[
H(\ell)^{(\reg(E))}\oplus H(\ell)[-1]^{(\reg(E))}\overset{I-D^*}\lra H(\ell)^{(E^0)}\oplus H(\ell)[-1]^{(E^0)}\to H(\cO_{A,B}^C)
\]
\end{thm}
\begin{proof}
Consider the Laurent polynomial algebra $L_1=\ell[t,t^{-1}]$. By Proposition \ref{prop:stable} and Theorem \ref{thm:cohnkk} we have a triangle in $kk$ of the form
\[
j(L_1^{(\reg(E))})\overset{f}\lra j(L_1^{(E^0)})\to j(\cO_{A,B}^C).
\]
Observe that $L_1$ is the Leavitt path algebra of the graph consisting of a single loop, hence by \cite{cm1}*{Theorem 5.4} we have $j(L_1\otimes R)=j(R)\oplus j(R)[-1]$ for every $R\in
\aha$. For each $v\in E^0$, $w\in \reg(E)$ and $i\in\{0,1\}$, let 
\[
\iota_{v,i}:j(\ell)[-i]\to j(\ell^{\reg(E)})[-i]\subset j(\ell^{\reg(E)})\oplus j(\ell^{\reg(E)})[-1]
\] 
be the inclusion and
\[
p_{v,i}:j(\ell^{(E^0)})\oplus j(\ell^{(E^0)})[-1]\to j(\ell)[-i]
\]
the projection in/onto the $(v,i)$-summand. Because $H$ is homotopy invariant, excisive, and $E$-stable, it factors as $H=\overline{H}\circ j$ with $\overline{H}$ a triangulated functor. Further, in view of the additivity hypothesis on $H$, $H(f)=\ol{H}(M)$ where $M$ is a matrix indexed by $(E^0\times\{0,1\})\times (\reg(E)\times\{0,1\})$, with coefficients
\[
M_{(w,i),(v,j)}=p_{(w,i)}\circ f\circ \iota_{(v,j)}\in kk(\ell[-j],\ell[-i]).
\]
The argument of \cite{cm2}*{Proposition 5.2} shows that 
\[
M_{(w,i),(v,0)}=\delta_{i,0}(\delta_{w,v}-A_{v,w})=D_{(v,0),(w,i)}.
\] 
To compute the coefficients $M_{(w,i),(v,1)}$, proceed as follows. Observe that if $F\subset E$ is any finite complete subgraph containing $s^{-1}\{v\}$, then 
\[
1_F-q_v+q_vt=(1_F-v+vt)(1_F-m_v+m_vt)^{-1}.
\]
Hence it suffices to establish that the following identity holds in $K_1(C(\Z,F,\phi_c))$ for any finite subgraph $F\subset E$ as above
\begin{equation}\label{eq:aproba}
[1_F-m_v+m_vt]=[\prod_{w\in F^0}(1_F-w+wt)^{B_{v,w}}(1_F-w+C_{v,w}w)]. 
\end{equation}\
Thus we may assume that $E=F$ is finite. Let $v,w\in E^0$ such that $A_{v,w}\ne 0$. 
Using the isomorphism $wC(\Z,F,\phi_c)^{A_{v,w}}\cong \bigoplus_{e\in s^{-1}\{v\}}ee^*C(\Z,F,\phi_c)$ as in the proof of Lemma \ref{lem:uvw}, we see that, with notation as in said lemma, for any $\lambda\in \cU(\ell)$ and $1\le i\le A_{v,w}-1$, we have 
\begin{equation}\label{eq:ellambda}
[1-w+\lambda w]=[1-m_{v,w}+\lambda e_ie_i^*+\sum_{j\ne i}e_je_j^*]\in K_1(C(\Z,F,\phi_c)).
\end{equation}
Set $\sigma_{v,w}=1-w+(-1)^{(A_{v,w}-1)}w$. By Lemma \ref{lem:uvw}
\begin{gather}\label{eq:eluvw}
 [(1-w+wt)^{B_{v,w}}]=[\sigma_{v,w}^{B_{v,w}}(1-m_{v,w}+u_{v,w})^{B_{v,w}}]\\
 =[\sigma_{v,w}^{B_{v,w}}(1-m_{v,w}+(-1)^{(A_{v,w}-1)B_{v,w}}C_{v,w}^{-1}t\cdot(e_0e_0^*)+\sum_{i=1}^{A_{v,w}-1}t\cdot(e_ie_i^*))]\nonumber   
\end{gather}
Next use \eqref{eq:ellambda} and \eqref{eq:eluvw} to compute that, in $K_1(C(\Z,F,\phi_c))$, we have
\begin{gather}\label{eq:rela}
[\prod_{w\in E^0}(1-w+wt)^{B_{v,w}}(1-w+C_{v,w}w)]=\\
[\prod_{w\in E^0}\sigma_{v,w}^{B_{v,w}}(1-m_{v,w}+t\cdot((-1)^{(A_{v,w}-1)B_{v,w}}(e_0e_0^*)+\sum_{i=1}^{A_{v,w}-1}e_ie_i^*))]=\nonumber\\
[\prod_{w\in E^0}(1-m_{v,w}+t\cdot m_{v,w})]=[1-m_v+t\cdot m_v]=[1-m_v+m_vt].\nonumber
\end{gather}
\end{proof}
Next we introduce notation and vocabulary that are used in the next two corollaries of Theorem \ref{thm:katkkh}. Consider the semi-direct product
\begin{equation}\label{def:w}
    \fW=\Z\ltimes \ell^*
\end{equation}
Recall that the cup-product makes $K_{\ge 0}(\ell)=\bigoplus_{n\ge 0}K_n(\ell)$ into a graded commutative ring, and $K_{\ge  n}(\ell)$ into an ideal for every $n\in\N_0$. 
Observe that the canonical map
\[
\fW\to K_0(\ell)\oplus K_1(\ell)=K_{\ge 0}(\ell)/K_{\ge 2}(\ell)
\]
is a ring homomorphism. Thus for every unital $R\in\aha$ and $n\in\Z$, the cup product  makes $K_n(R)\oplus K_{n-1}(R)$ into a graded $\fW$-module.
We shall also consider the following $\fW$-modules. 
\begin{align}\label{eq:bfabc}
\BF(A,B,C)=&\coker(I-D^*:\fW^{\reg(E)}\to \fW^{E^0})\\
{\vBF}(A,B,C)=&\coker(I-D:\fW^{E^0}\to \fW^{\reg(E)})\nonumber.
\end{align}
\begin{coro}\label{coro:katkh}
Let $R\in\aha$ and $n\in\Z$. If $R$ is flat over $\ell$, then there is an exact sequence
\begin{multline*}
0\to \BF(A,B,C)\otimes_{\fW}(KH_{n}(R)\oplus KH_{n-1}(R))\to KH_n(\cO_{A,B}^C\otimes R)\\
\to \ker((I-D^*)\otimes_\fW(KH_{n-1}(R)\oplus KH_{n-2}(R))\to 0.    
\end{multline*}

\end{coro}
\begin{proof} The homotopy $K$-theory spectrum defines a functor $KH:\aha\to\ho(\spt)$ to the homotopy category of spectra. Applying Theorem \ref{thm:katkkh} to the functor $KH(-\otimes R)$ and taking stable homotopy groups one gets a long exact sequence that can be expressed as the collection of the exact sequences of the corollary  for $n\in\Z$.
\end{proof}

\begin{rem}\label{rem:katkh}
The flatness of $R$ is needed in Corollary \ref{coro:katkh} to guarantee that $\otimes R$ preserves algebra extensions and so extends to a triangulated functor $kk\to kk$. One may also consider a variant of $j$, $j^s:\aha\to kk^s$  that is excisive only with respect to those extensions that admit an $\ell$-linear splitting, so that $\otimes R$ extends to $kk$ for any $R$. The analogue of Corollary \ref{coro:katkh} then holds for $j^s$ whenever the Cohn extension is $\ell$-linearly split, e.g. when the conditions of Corollary \ref{coro:basisunion} are satisfied.
\end{rem}

\begin{coro}\label{coro:katkkh}
Assume that $E^0$ is finite. Then there is a distinguished triangle in $kk$
\[
j(\ell)^{\reg(E)}\oplus j(\ell)[-1]^{reg(E)}\overset{I-D^*}\lra j(\ell)^{E^0}\oplus j(\ell)[-1]^{E^0}\to j(\cO_{A,B}^C).
\]
\end{coro}
\begin{proof} Apply Theorem \ref{thm:katkkh} to the functor $j:\aha\to kk$.
\end{proof}

\begin{coro}
Let $R\in\aha$. There is an exact sequence
\begin{multline*}
0\to{\vBF}(A,B,C)\otimes_{\fW}(KH_{2}(R)\oplus KH_1(R))\to kk(\cO_{A,B}^C,R)\\
\to\hom_{\fW}(\BF(A,B,C),KH_1(R)\oplus KH_0(R))\to 0.    
\end{multline*}

\end{coro}

\begin{thm}\label{thm:katkat}
Let $n\ge 1$, $M,N\in M_n(\Z)$ and $P\in M_n(\cU(\ell))$, and let $R\in\aha$ such that there is a distinguished triangle in $kk$
\begin{equation}\label{tri:kat}
\xymatrix{j(\ell)^n\oplus j(\ell)[-1]^n\ar[rr]^{\begin{bmatrix}M&P\\ 0&N\end{bmatrix}}&& j(\ell)^n\oplus j(\ell)[-1]^n\ar[r]& j(R)}.
\end{equation}
Then there exist matrices $A\in M_{2n}(\N_0)$, $B\in M_{2n}(\Z)$ and $C\in M_{2n}(\cU(\ell))$ such that $(A,B)$ is KSPI and an isomorphism
\[
j(\cO_{A,B}^C)\cong j(R).
\] 
\end{thm}

\begin{proof} Let $I_m\in M_m(\Z)$ be the identity matrix. The matrix 
\[
E=\begin{bmatrix}M& 0& P & 0\\
0& I_n&0&0\\
0&0&N&0\\
0&0&0&I_n\end{bmatrix}
\]
defines an endomorphism of $j(\ell)^n\oplus j(\ell)[-1]^n$ in $kk$, and 
the direct sum of \eqref{tri:kat} with the trivial triangle
\[
\xymatrix{j(\ell)^n\oplus j(\ell)[-1]^n\ar[r]^{I_{2n}}& j(\ell)^n\oplus j(\ell)[-1]^n\ar[r]&0}
\]
is isomorphic to a triangle
\begin{equation}\label{tri:kat2}
\xymatrix{j(\ell)^{2n}\oplus j(\ell)^{2n}\ar[r]^{E}& j(\ell)^{2n}\oplus j(\ell)^{2n}\ar[r]&j(R).}
\end{equation}

Define matrices $X$ and $Y$ in $M_n(\Z)$ as in the proof of \cite{kat}*{Lemma 3.1}, taking into account that the matrices named $A'$ and $B'$ are what we call $M$ and $N$ here.  Let
\[
U=\begin{bmatrix}I_n& I_n& 0\\
                  0& I_n& 0\\
									0& 0& I_{2n}\end{bmatrix},\ \ V=\begin{bmatrix}0&-I_n&0&0\\ I_n&Y&0&0\\
																																 0 & 0 &0&-I_n\\
																																 0&0& -I_n&0\end{bmatrix}\in M_{2n}(\Z).
\]
Define $A,B\in M_{2n}(\Z)$ and $C\in M_{2n}(\cU(\ell))$ by the identity
\[
\begin{bmatrix}I_{2n}-A^t& C^*\\
               0& I_{2n}-B^t\end{bmatrix}=UEV.
\]
Observe that the matrices $A^t$ and $B^t$ are called $A$ and $B$ in \cite{kat}*{Lemma 3.1}; in particular $(A,B)$ is KSPI. Moreover, the matrices $U$ and $V$ induce an isomorphism of triangles between \eqref{tri:kat2} and the triangle of Corollary \ref{coro:katkkh} associated to the matrices $A,B$ and $C$ we have just defined. In particular, $j(R)\cong j(\cO_{A,B}^C)$. 
\end{proof}

\begin{coro}\label{coro:katkat} 
Let $\ell$ be either a field or a PID and let $R\in\aha$ such that there is a distinguished triangle in $kk$
\[
\xymatrix{(j(\ell)\oplus j(\ell)[-1])^n\ar[r]^f& (j(\ell)\oplus j(\ell)[-1])^n\ar[r]& j(R).}
\]
Then there exist matrices $A,B,C$ such that $(A,B)$ is KSPI and $j(R)\cong j(\cO_{A,B}^C)$. If moreover $\ell\supset\Q$ is a field, or if $\ell$ is any field and the pair $(A,B)$ satisfies the condition of Proposition \ref{prop:katsimp}, then $\cO_{A,B}^C$ is simple purely infinite. 
\end{coro}
\begin{proof}
If $\ell$ is a field or a $PID$, $kk(\ell,\ell)=kk(\ell[-1],\ell[-1])=K_0(\ell)=\Z$, $kk(\ell[-1],\ell)=K_1(\ell)=\ell^*$ and $kk(\ell,\ell[-1])=K_{-1}(\ell)=0$. Thus $f$ must be as in Theorem \ref{thm:katkat}. Thus the first assertion follows from said theorem. The second assertion follows from Theorem \ref{thm:katpis} and Proposition \ref{prop:katsimp}. 
\end{proof}
\section{\topdf{$K$}{K}-regularity via the twisted Laurent polynomial picture}\label{sec:twistlaure}
\numberwithin{equation}{section}
Let $E$ be a row-finite graph and let $(G,E,\phi_c)$ be a twisted $EP$-tuple. Consider the $\Z$-grading $L(G,E,\phi_c)=\bigoplus_{n\in\Z}L(G,E,\phi_c)_n $,
\[
L(G,E,\phi_c)_n=\mspan\{\alpha g\beta^*: |\alpha|-|\beta|=n\}.
\] 
For each $n\ge 0$, let 
\begin{multline*}
L_{0,n}(G,E,\phi_c)=\mspan\{\alpha g\beta^*:r(\alpha)=g(r(\beta))\in\sink(E),\ \ |\alpha|=|\beta|\le n\}\\
+\mspan\{\alpha g\beta^*:r(\alpha)=g(r(\beta))\in\reg(E),\ \ |\alpha|=|\beta|\le n\}.
\end{multline*}
We have $L(G,E,\phi_c)_0=\bigcup_{n\ge 0}L_{0,n}(G,E,\phi_c)$. Observe that setting $G=\{1\}$ we recover the usual grading of $L(E)$ and the usual filtration of $L(E)_0$.  Like in the Leavitt path algebra case \cite{abc}*{Section 5} when $E$ is finite without sources, we may also regard $L(G,E,\phi_c)$ as a corner skew Laurent polynomial ring in the sense of \cite{fracskewmon}, as follows. Pick an edge $e_v\in r^{-1}\{v\}$ for each $v\in E^0$. Then the elements of $L(G,E,\phi_c)$
\[
t_+=\sum_{v\in E^0}e_v, \text{ and } t_-=t_+^*
\] 
are homogeneous of degrees $1$ and $-1$, respectively, and satisfy $t_-t_+=1$. Hence 
\[
p=\psi(1)\in L(G,E,\phi_c)_0.
\]
is an idempotent, and the map 
\begin{equation}\label{map:corneriso}
\psi:L(G,E,\phi_c)_0\to L(G,E,\phi_c)_0,\ \ \psi(x)=t_+xt_-
\end{equation}
is an algebra monomorphism; its image is the corner $pL(G,E,\phi_c)_0p$. Hence by \cite{fracskewmon}*{Lemma 2.4} we may regard $L(G,E,\phi_c)$ as a Laurent polynomial ring twisted by the corner isomorphism $\psi$; using the notation of \cite{fracskewmon}, we write
\begin{equation}\label{eq:lppic}
L(G,E,\phi_c)=L(G,E,\phi_c)_0[t_+,t_-;\psi]. 
\end{equation}
In the next Lemma we use the following notation. For a graph $E$, a vertex $v\in E^0$ and $m\ge 0$, we put
\[
\cP_{v,m}=\{\alpha\in\cP(E)_v\colon |\alpha|=m\}.
\]
\begin{stan}
 From this point on, in all (twisted) $EP$-triples we consider, the group $G$ will be assumed to act trivially on $E^0$. 
\end{stan}

\begin{nota}
Let $v\in E^0$; regard $v$ as an element of $L(G,E,\phi_c)$. Write
\[
I_v=\ann_{\ell[G]}(v),\, R_v=\ell[G]/I_v
\]
for the annihilator ideal and the quotient ring. 
\end{nota}

\begin{lem}\label{lem:L0n}
Let $(G,E,\phi_c)$ be a twisted EP-tuple with $E$ row-finite. Assume that $G$ acts trivially on $E^0$. Put
\[
\cM_n(G,E):=(\bigoplus_{v\in\sink(E)}\bigoplus_{m=0}^nM_{\cP_{v,m}}\ell[G])\oplus (\bigoplus_{v\in\reg(E)}M_{\cP_{v,n}}\ell[G])
\]
\item[i)] There is a surjective algebra homomorphism 
\[
 \pi_n:\cM_n(G,E)
\onto L(G,E,\phi_c)_{0,n}, \ \
\pi_n(\epsilon_{\alpha,\beta}g)=\alpha g\beta^*.
\]
\item[ii)] $I_v=0$ for all $v\in\sink(E)$ and 
\[
\ker(\pi_n)=\bigoplus_{v\in\reg(E)}M_{\cP_{v,n}}I_v.
\]
Hence we have
\[
L(G,E,\phi_c)_{0,n}\cong (\bigoplus_{v\in\sink(E)}\bigoplus_{m=0}^nM_{\cP_{v,m}}R_v)\oplus (\bigoplus_{v\in\reg(E)}M_{\cP_{v,n}}R_v).
\]
\item[iii)] Under the isomorphism above the inclusion $L(G,E,\phi_c)_{0,n}\subset L(G,E,\phi_c)_{0,n+1}$ identifies with the inclusion
\[
\bigoplus_{v\in\sink(E)}\bigoplus_{m=0}^nM_{\cP_{v,m}}R_v\subset \bigoplus_{v\in\sink(E)}\bigoplus_{m=0}^{n+1}M_{\cP_{v,m}}R_v
\]
on the first summand and is induced by the map 
\begin{gather*}
\bigoplus_{v\in\reg(E)}M_{\cP_{v,n}}R_v\to (\bigoplus_{v\in\sink(E)}M_{\cP(n+1,v)}R_v)\oplus (\bigoplus_{v\in\reg(E)}M_{\cP(n+1,v)}R_v)\\
\epsilon_{\alpha,\beta}g\mapsto \sum_{s(e)=r(\alpha)}\epsilon_{\alpha g(e),\beta e}\phi_c(g,e).
\end{gather*}
on the second. 
\end{lem}
\begin{proof} It is clear that the $\ell$-linear map $\pi_n$ is surjective. Observe that  
\begin{equation}\label{eq:prod=delta}
\alpha,\beta\in \left(\bigcup_{v\in\sink(E), 1\le m\le n}\cP_{v,m}\cup\bigcup_{v\in\reg(E)}\cP_{v,n}\right)\Rightarrow \beta^*\alpha=\delta_{\alpha,\beta}r(\alpha).
\end{equation}
It follows from this that $\pi_n$ is a ring homomorphism, proving i). Next let $x=\sum_{\alpha,\beta} \epsilon_{\alpha,\beta}x_{\alpha,\beta}\in\ker(\pi_n)$. Then 
\begin{equation}\label{eq:alphabeta}
\pi_n(x)=\sum_{\alpha,\beta}\alpha x_{\alpha,\beta}\beta^*=0.
\end{equation} 
Taking \eqref{eq:prod=delta} into account and  multiplying \eqref{eq:alphabeta} by $\alpha^*$ on the left and by $\beta$ on the right, we obtain that \eqref{eq:alphabeta} holds if and only if
\[
x_{\alpha,\beta}\in I_{r(\alpha)} 
\]
holds in $L(G,E,\phi_c)$ for all $\alpha,\beta$ for which $x_{\alpha,\beta}$ is defined. Observe that if we regard a vertex $v\in E^0$ as an element of $C(G,E,\phi_c)$, then 
$I_v=\{x\in\ell[G]\,\colon\, vx\in\cK(G,E,\phi_c)\}$. Recall $\cK(G,E,\phi_c)$ is spanned by the elements
\[
\alpha q_wg\beta^*=\alpha g\beta^*-\sum_{s(e)=w}\alpha e\phi(g,g^{-1}(e))(\beta g^{-1}(e))^*
\]
with $w\in\reg(E)$. The above is the unique expression of the element of the left as a linear combination of the the basis $\cB$ of $C(G,E,\phi_c)$, and contains no basis elements of the form $vh$ with $v\in\sink(E)$ and $h\in G$. This proves that $I_v=0$ for every $v\in\sink(E)$.
Thus part ii) is proved. Part iii) is a straightforward application of the Cuntz-Krieger relation $CK2$. 
\end{proof}

Let $v\in\reg(E)$ and $w\in E^0$ such that $A_{v,w}\ne 0$; set 
\begin{equation}\label{eq:xvw}
X_{v,w}=\bigoplus_{e\in vE^1w}\ell e\otimes R_w
\end{equation}
Observe that the $\ell$-linear map
\begin{equation}\label{map:xtolge}
X_{v,w}\to L(G,E,\phi_c),\ \ e\otimes x\mapsto ex
\end{equation}
is injective, and that its image is a left $R_v$-submodule. In Lemma \ref{lem:flat1} below we regard $X_{v,w}$ as a left $R_v$-module via the the map \eqref{map:xtolge}. 
 
\begin{lem}\label{lem:flat1}
Let $E$ be a finite graph with incidence matrix $A$. Assume that $G$ acts trivially on $E^0$. The following are equivalent for $n\ge 0$. 
\item[i)] $L(G,E,\phi_c)_{0,n+1}$ is a flat left $L(G,E,\phi_c)_{0,n}$-module. 
\item[ii)] For every $(v,w)\in \reg(E)\times E^0$ such that $A_{v,w}\ne 0$, the left $R_v$-module \eqref{eq:xvw} is flat. 
\end{lem}
\begin{proof} If $M$ is a left module over a unital ring $R$, $m\ge 1$ and $p_1,\dots,p_m\in R$ are central orthogonal idempotents such that $\sum_{i=1}^mp_i=1$, then $M$ is flat over $R$ if and only if $p_iM$ is a flat $ Rp_i$-module for all $i$. We apply this with $R=L_{0,n}$, $M=L_{0,n+1}$ and the orthogonal idempotents that correspond to the identity matrices of each of the matrix algebras in the direct sum decomposition of Lemma \ref{lem:L0n} and using the identification of the inclusion $L(G,E,\phi_c)_{0,n}\subset L(G,E,\phi_c)_{0,n+1}$ given therein. If $v\in\sink(E)$ and $m\le n$, then under the identification of Lemma \ref{lem:L0n}, $\id_{\cP_{v,m}}\cdot L(G,E,\phi_c)_{0,n+1}=M_{\cP_{v,m}}R_v$, which is flat over itself. If $v\in \reg(E)$, then 
\[
\id_{\cP_{v,n}}\cdot L(G,E,\phi_c)_{0,n+1}=\bigoplus_{\{w|A_{v,w}\ne 0\}}\sum_{\{e\in vE^1w, \alpha\in \cP_{v,n}\}}\epsilon_{\alpha e,\alpha e}M_{\cP(w,n+1)}R_w.
\]
One checks that each of the summands 
\begin{equation}\label{eq:elmodu}
\sum_{\{e\in vE^1w, \alpha\in \cP_{v,n}\}}\epsilon_{\alpha e,\alpha e}M_{\cP(w,n+1)}R_w
\end{equation}
is a left $M_{\cP_{v,n}}R_v$-submodule. Hence $L(G,E,\phi_c)_{0,n}\subset L(G,E,\phi_c)_{0,n+1}$ is flat if and only if \eqref{eq:elmodu} is flat for every $w\in E^0$ such that $A_{v,w}\ne 0$. Moreover \eqref{eq:elmodu} decomposes as direct sum, indexed by $\gamma\in\cP(w,n+1)$, of the $\cM_{P(n,v)}R_v$-submodules
\[
X_{v,w,\gamma}=\bigoplus_{\{\alpha\in\cP_{v,n}, e\in vE^1w\}}\epsilon_{\alpha e,\gamma}R_w. 
\]
So again the flatness of $L(G,E,\phi_c)_{0,n+1}$ over $L(G,E,\phi_c)_{0,n}$ boils down to that of each of the $X_{v,w,\gamma}$. Equip $\ell^{\cP(v,n)}=\ell^{\cP(v,n)\times\{1\}}$ with its canonical left $M_{\cP_{v,n}}$-module structure and view $\ell^{\cP(v,n)}\otimes X_{v,w}$  as a module over $M_{\cP_{v,n}}R_v=M_{\cP_{v,n}}\otimes R_v$ in the obvious way. One checks that
\[
X_{v,w,\gamma}\to \ell^{\cP_{v,n}}\otimes X_{v,w},\ \ \epsilon_{\alpha e,\gamma}x\mapsto \alpha\otimes e\otimes x 
\] 
is an isomorphism of left $M_{\cP_{v,n}}R_v$-modules. Since the $M_{\cP_{v,n}}$-module $\ell^{\cP_{v,n}}$ is projective, whence flat, we get that $X_{v,w,\gamma}$ is flat over $M_{\cP_{v,n}}\otimes R_v$ if and only if $X_{v,w}$ is flat over $R_v$. This concludes the proof. 
\end{proof}

\begin{lem}\label{lem:iv=iw=0}
Let $(G,E,\phi_c)$ be a twisted $EP$-tuple satisfying the conditions of Lemma  \ref{lem:flat1} and let $(v,w)\in\reg(E)\times E^0$. Assume that $I_v=I_w=0$. Then $X_{v,w}$ is a flat left $R_v=\ell[G]$-module if and only if $\ell[G]/\ann_{\ell[G]}(e)$ is $\ell[G]$-flat for all $e\in vE^1w$.
\end{lem}
\begin{proof}
For $g\in G$ and $(e,h)\in vE^1w\times G$, set
\begin{equation}\label{map:actphi}
g\cdot(e,h)=(g(e), \phi(g,e)h).
\end{equation}
One checks, using that $\phi$ is a cocycle, that \eqref{map:actphi} defines a left action of $G$ on $vE^1w\times G$, and thus a linear $G$-action on 
\[
\ell[vE^1w\times G]=X_{v,w}.
\]
For the identification above we have used our hypothesis that $I_w=0$. The decomposition of $vE^1w\times G$ into $G$-orbits gives a corresponding direct sum decomposition of $R_v=\ell[G]$ -modules 
\[X_{v,w}=\bigoplus_{K\in (vE^1w\times G)/G} \ell[K] \]
Let $e\in vE^1w$, $h\in G$, and $K_{e,h}$ the orbit of $(e,h)$. Observe that right multiplication by $h$ gives an $\ell[G]$-module isomorphism $\ell[K_{e,1}]\cong \ell[K_{e,h}]$. Moreover we have an isomorphism of left $\ell[G]$-modules
\begin{equation}\label{map:flatquot}
\ell[G]/\ann_{\ell[G]}(e)\cong \ell[K_{(e,1)}].
\end{equation}
Summing up, $X_{v,w}$ is flat if and only if \eqref{map:flatquot} is flat for all $e\in vE^1w$, as we had to prove.
\end{proof}
 
\begin{thm}\label{thm:Kreg}
Let $(G,E,\phi_c)$ be a twisted $EP$-tuple. Assume that $E$ is row-finite and that $G$ acts trivially on $E^0$. Further assume that 
$R_v$ is regular supercoherent for every $v\in E^0$ and that condition (ii) of Lemma \ref{lem:flat1} is satisfied. Then $L(G,E,\phi_c)$ is $K$-regular.
\end{thm}
\begin{proof} Assume first that $E$ is finite without sources. Put $S=L(G,E,\phi_c)_0$; let $\psi:S\to S$ be the corner isomorphism of \eqref{map:corneriso} and $B=S[\psi^{-1}]$ the colimit of the inductive system 
\[
\xymatrix{S\ar[r]^\psi&S\ar[r]^\psi&S\ar[r]^\psi&\dots},
\]
$\tilde{B}=B\oplus\ell$ its unitalization and $\tilde{\psi}:\tilde{B}\to\tilde{B}$ the induced unital automorphism. Let $NK_*(S,\psi)_\pm=NK_*(\tilde{B},\tilde{\psi})$, the twisted nil-$K$-theory groups. Note that because $B$ is a filtering colimit of unital rings, it satisfies excision in $K$-theory, and thus these nil-$K$-groups are the same as those defined in \cite{abc}*{Notation 3.4.1}. By \eqref{eq:lppic} and \cite{abc}*{Theorems 3.6 and 8.4} the comparison map $K_*(L(G,E,\phi_c))\to KH_*(L(G,E,\phi_c))$ fits into a map of long exact sequences
\begin{equation}
\xymatrix{K_n(S)\ar[r]^(.25){1-\psi}\ar[d]& K_n(S)\oplus NK_n(S,\psi)_+\oplus NK_n(S,\psi)_-\ar[d]\ar[r]&K_n(L(G,E,\phi_c))\ar[d]\ar[r]&K_{n-1}(S)\ar[d]\\
KH_n(S)\ar[r]^{1-\psi}& KH_n(S)\ar[r]&KH_n(L(G,E,\phi_c))\ar[r]&KH_{n-1}(S)}
\end{equation}
Because by hypothesis, the $R_v$ are regular supercoherent for all $v$, so is any finite sum of finite matrix rings over them; in particular $S_n=L(G,E,\phi_c)_{0,n}$ is regular supercoherent for all $n$. By \cite{gersten}*{Proposition 1.6}, the colimit of an inductive system of regular supercoherent rings with unital flat transition maps is regular supercoherent.  Hence $S$ is regular supercoherent by Lemma \ref{lem:flat1}. In particular the comparison map $K_n(S[t_1,\dots,t_p])\to KH_n(S[t_1,\dots,t_p])$ is an isomorphism for all $n$ and $p$.  Now the argument of \cite{abc}*{Proposition 7.1} applies verbatim to show that $\tilde{B}$ is regular supercoherent. Thus $NK_*(S,\psi)_\pm=NK_*(\tilde{B},\tilde{\psi})_\pm=0$ by \cite{abc}*{Lemma 7.2}. It follows that the comparison map $K_*(L(G,E,\phi_c))\to KH_*(L(G,E,\phi_c))$ is an isomorphism. Substituting $\ell[t_1,\dots,t_m]$ for $\ell$ we get that $K_*(L(G,E, \phi_c)[t_1,\dots,t_m])\to KH_*(L(G,E,\phi_c)[t_1,\dots,t_m])$
is an isomorphism for all $m$. This proves that $L(G,E,\phi_c)$ is $K$-regular whenever $E$ is finite without sources and the hypothesis of the theorem are satisfied. Let now $E$ be a finite graph, $v\in E^0$ a source, and $E_{|v}$ the graph obtained from $E$ upon removing $v$. Then $1-v$ is a full idempotent of $L(E)$ and therefore also of $L(G,E,\phi_c)$. Furthermore, the corner $C=(1-v)L(G,E,\phi_c)(1-v)$ is the span of elements $\alpha g\beta^*$ with $g\in G$ and $\alpha,\beta\in\cP(E_{|v})$; it follows that $C\cong L(G,E_{|v},\phi_c{|E_{|v}})$ is the algebra of the $EP$-tuple obtained from $(G,E,\phi_c)$ upon restricting $\phi$ and $c$ to $G\times E_{|v}^1$. Repeating this process a finite number of times we end up with a finite graph $E'$ without sources such that $L(G,E',\phi_c{|E'})$ is isomorphic to a full corner of $L(G,E,\phi_c)$. By what we have already proved, $L(G,E', \phi_{|E'})$ is $K$-regular; therefore the same is true of $L(G,E,\phi_c)$, since the latter is Morita equivalent to the former. This proves the theorem for finite $E$. If now $E$ is any row-finite graph and $F\subset E$ a finite complete subgraph, then the algebra $L(G,F,\phi_c{|F})$ of the restriction $EP$-tuple is isomorphic to the subalgebra of $L(G,E,\phi_c)$  linearly spanned by the elements $\alpha g\beta^*$ with $\alpha,\beta\in F$ and $g\in G$. Hence $L(G,E,\phi_c)=\colim_F L(G,F,{\phi_c}{|F})$ where the colimit runs over all finite complete subgraphs. Thus $L(G,E,\phi_c)$ is $K$-regular because each $L(G,F,{\phi_c}{|F})$ is, by what we have already proved. 
\end{proof}

\begin{coro}\label{coro:kregpseudo}
Let $(G,E,\phi_c)$ be a twisted $EP$-tuple as in Theorem \ref{thm:Kreg}. If $(G,E,\phi_c)$ is pseudo-free, then $L(G,E,\phi_c)$ is $K$-regular. 
\end{coro}
\begin{proof} Because $(G,E,\phi_c)$ is pseudo-free, we have $I_v=0$ for all $v\in E^0$, by Corollary \ref{coro:basisunion}. Hence in view of Lemma \ref{lem:iv=iw=0} it suffices to show that
$\ann_{\ell[G]}(e)=0$ for all $e\in E^1$. Let $e\in E^1$ and $K=K_{(e,1)}$ as in the proof of Lemma \ref{lem:iv=iw=0}. It follows from Corollary \ref{coro:basisunion} and the identity \eqref{eq:b''pseudo} that $x=\sum_{g\in G}\lambda_gg \in \ann_{\ell[G]}(e)$ if and only if for every $(f,h)\in K$ we have
\begin{equation}\label{eq:escero}
0=\sum_{\{g: g(e),\phi(g,e)=(f,h)\}}\lambda_{g}c(g,e).    
\end{equation}
Next observe that if $x$ satisfies \eqref{eq:escero} then it annihilates $ee^*\in C(G,E,\phi_c)$. By Lemma \ref{lem:ann=0} and pseudo-freeness, this implies that $x=0$. Hence $\ann_{\ell[G](e)=0}$ for all $e\in E^1$, concluding the proof. 
\end{proof}

\begin{ex}\label{ex:unflat}
Let $\ell$ be a Noetherian domain and let $(\Z,E,\phi_c)$ be a twisted EP-tuple where $\Z$ acts trivially on $E^0$. Assume $(\Z,E,\phi_c)$ is partially pseudo-free but not pseudo-free. Write $L=\ell[t,t^{-1}]$; then \eqref{eq:b''pseudo} maps injectively to a basis of $L(\Z,E,\phi_c)$. Since the latter basis contains $vg$ for all $g\in G$, they are $\ell$-linearly independent, so we have $R_v=L$ for all $v\in E^0$. By Lemma \ref{lem:ann=0} there is an edge $e$ and a nonzero element $x\in L$ that annihilates $ee^*\in C(\Z,E,\phi_c)$. Then $x$ also annihilates $e\in L(\Z,E,\phi_c)$, so $J:=\ann_{L}(e)\ne 0$. Moreover $J\ne L$, by Proposition \ref{prop:basisunion}. Because we are assuming that $\ell$ is Noetherian, the same is true of $L$, and thus $L/J$ is flat if and only if it is projective, by Villamayor's theorem. Hence the flatness of $L/J$ would imply that $L$ is a decomposable $L$-module, which would contradict the hypothesis that $\ell$ is a domain. So $L/J$ is not flat, and thus neither is $X_{s(e),r(e)}$, by Lemma \ref{lem:iv=iw=0}. By  Lemma \ref{lem:flat1}, this implies that $L(\Z,E,\phi_c)_{0,n+1}$ is not  flat over $L(\Z,E,\phi_c)_{0,n}$ for any $n\ge 0$.
\end{ex}

\section{\topdf{$K$}{K}-regularity of twisted Katsura algebras}
\begin{lem}\label{lem:flat2} Let $(A,B,C)$ and $E$ be as in Section \ref{sec:kat} above and let $(v,w)\in\reg(E)\times E^0$ such that $A_{v,w}\ne 0$. Let 
$X_{v,w}$ be the $R_v$-module of \eqref{eq:xvw}.
\item[i)] If $B_{v,w}\ne 0$, then $X_{v,w}$ is flat.
\item[ii)] If $B_{v,w}=0$ and there exists $w'$ such that $B_{v,w'}\ne 0$, then $X_{v,w}$ is not flat.
\item[iii)] Suppose $B_{v,w}=0$ for all $w\in r(s^{-1}\{v\})$. Futher assume that $C_{v,w}-C_{v,w'}$ and $C_{v,w}-1$ are either zero or invertible in $\ell$ for all $w, w'\in r(s^{-1}\{v\})$. Then $X_{v,w}$ is flat for all $w\in r(s^{-1}\{v\})$.
\end{lem}
\begin{proof}
Set $a=A_{v,w}$, $b=B_{v,w}$, $c=C_{v,w}$, $X=X_{v,w}$ and $L=\ell[t,t^{-1}]$. For each $0\le i<a$, write $e_i$ for the unique edge $e\in vE^1w$ with $\n(e)=i$. If $b=0$, then $te_0\otimes x=ce_0\otimes x$ 
and $te_i\otimes x=e_i\otimes x$ for all $x\in L$ and $1\le i<a$. Hence the monic polynomial $f(t)=(t-1)(t-c)$ annihilates $X$. Because $f(t)$ is not a zero divisor in $L$, 
\[
L\overset{f(t)}{\lra}L
\]
is a free resolution of $L/f(t)L$. Hence $\tor_1^L(L/(t-1)(t-c),X)=X\ne 0$. In particular, $X$ is not flat over $L$. In the situation of ii), $I_v=0$, so $R_v=L$, and thus $X$ is not flat over $R_v$. In the situation of iii),
If $b\ne 0$, let $c_1,\dots,c_r$ be the distinct elements of the set $\{1\}\cup\{C_{v,w}:w\in r(s^{-1}(\{v\}))\}$. Then
\[
R_v\cong L/\prod_{i=1}^r(t-c_i)L\cong \bigoplus_{i=1}^r\ell[t]/(t-c_i)\ell[t],
\]
and $X$ is a direct sum of copies of some of the summands in the decompositon above. Hence it is a projective $R_v$-module and in particular it is flat.

Next assume that $b\ne 0$; note that $R_v=L$ in this case. Let $r:\Z\to\{0,\dots,a-1\}$ and $q:\Z\to\Z$ be the remainder and the quotient function in the division by $a$. Let $L_{b,c}$ be the $\ell$-module $L$ equipped with the following $\Z$-action 
\[
t\cdot_bx=ct^bx.
\]
Then $L_{b,c}$ is free with basis $1,t,\dots,t^{|b|-1}$ and 
\[
L_{b,c}\to X\ \ t^n\mapsto e_{r(n)}\otimes t^{q(n)}
\]
is an isomorphism of left $L$-modules. Hence $X$ is a free left $L$-module. 
\end{proof}

\begin{prop}\label{prop:KatKreg}
Let $\ell$ be regular supercoherent and $(A,B,C)$ a twisted Katsura triple. If either of the following holds  then $\cO_{A,B}^C$ is $K$-regular. 
\item[i)] $B_{v,w}=0\iff A_{v,w}=0$.
\item[ii)] $\ell$ is a field and if $v\in\reg(E)$ is such that $B_{v,w}=0$ for some $w\in r(s^{-1}\{v\})$, then $B_{v,w'}=0$ for all
$w'\in r(s^{-1}\{v\})$. 
\end{prop}
\begin{proof} Immediate from Lemma \ref{lem:flat2} and Theorem \ref{thm:Kreg}.
\end{proof}
\begin{rem}\label{rem:KatKreg} By \cite{ep}*{Lemma 18.5} condition i) of Proposition \ref{prop:KatKreg} is equivalent to the condition that the $EP$-tuple $(\Z,E,\phi)$ associated to $(A,B)$ in Section \ref{sec:kat} be pseudo-free. If condition ii) of the same proposition is satisfied, then any vertex $v$ with $B_{v,w}=0$ for some $w\in r(s^{-1}\{v\})$ is a \emph{$B$-sink} in the sense of \cite{nyort}*{page 2248}.
\end{rem}

\begin{coro}\label{coro:KatKreg}
Assume either of the conditions of Proposition \ref{prop:KatKreg} is satisfied. Then for any $n\in\Z$ there is a short exact sequence
\begin{multline*}
0\to \BF(A,B,C)\otimes_{\fW}(K_{n}(\ell)\oplus K_{n-1}(\ell))\to K_n(\cO_{A,B}^C)\\
\to \ker((I-D^*)\otimes_\fW(K_{n-1}(\ell)\oplus KH_{n-2}(\ell)\to 0.    
\end{multline*}
\end{coro}
\begin{proof} Immediate from Proposition \ref{prop:KatKreg} and Corollary \ref{coro:katkh}.
    
\end{proof}
\section{Ring theoretic regularity: the universal localization picture}\label{sec:uniloc}
Let $(G,E,\phi_c)$ be a twisted EP-tuple. In this section we assume that $E$ is finite and that $G$ acts trivially on $E^0$. We show that  $L(G,E,\phi_c)$ can be interpreted as the universal localization of an algebra $P(G,E,\phi_c)$. We use this to give sufficient conditions that guarantee that $L(G,E,\phi_c)$
is \emph{regular} in the sense that every (right) $S$-module has finite projective dimension, and apply this to the case of twisted Katsura algebras.

For each $(v,w)\in\reg(E)\times E^0$, let $X_{v,w}$ be as in \eqref{eq:xvw}. Regard $X_{v,w}$ as an $(R_v,R_w)$-bimodule, with the left $R_v$-module structure induced via the map \eqref{map:xtolge} and the obvious right $R_w$-module structure. Set 
\begin{equation}\label{R&X}
R=\bigoplus_{v\in E^0}R_v,\ \ X=\bigoplus_{v,w}X_{v,w}.
\end{equation}
Observe that $X$ is an $R$-bimodule; let 
\begin{equation}
P(G,E,\phi_c):=T_RX
\end{equation}
be the tensor algebra. Since by definition $R$ and $X$ are a subalgebra and an $R$-sub-bimodule of $L(G,E,\phi_c)$, we have a canonical algebra homomorphism
\begin{equation}\label{map:pgetolge}
P(G,E,\phi_c)\to L(G,E,\phi_c)
\end{equation}

\begin{lem}\label{lem:pgetolge}
The homomorphism \eqref{map:pgetolge} is injective. 
\end{lem}
\begin{proof}
By definition, $R$ and $X$ are included in $L(G,E,\phi_c)$. Moreover, the $n$-fold tensor product $T^n_RX$ is freely generated as a right $R$-module by the paths of length $n$ in $E$. Thus it suffices to show that if $\cF\subset\cP(E)$ is a finite subset and $x_{\alpha}\in R_{r(\alpha)}$ for all $\alpha\in\cF$, then
\begin{equation}\label{eq:pgetolge}
\sum_{\alpha}\alpha x_\alpha=0
\end{equation}
implies that each $x_\alpha=0$. Choose a lift $y_{\alpha}\in\ell[G]$ for each $\alpha\in\cF$; then \eqref{eq:pgetolge} implies that $\sum_{\alpha\in\cF}\alpha y_\alpha\in \cK(G,E,\phi_c)$. If follows from this and the fact that $\cB$ and $\cB'$ are $\ell$-linear basis of $C(G,E,\phi_c)$ and $\cK(G,E,\phi_c)$, that $y_{\alpha}\in I_{r(\alpha)}$ so that $x_{\alpha,\beta}=0$ for each $\alpha\in\cF$.
\end{proof}
                                                  
For each regular vertex $v$, let
\[
\sigma_v:\bigoplus_{s(e)=v}r(e)P(G,E,\phi_c)\to vP(G,E,\phi_c),\ \ \sigma_v(r(e)a)=ea.
\]
Let $\Sigma=\{\sigma_v:v\in \reg(E)\}$ and let $P(G,E,\phi_c)_\Sigma$ be the universal localization. Observe that scalar extension along \eqref{map:pgetolge} inverts the elements of $\Sigma$; hence we have a canonical algebra homomorphism
\begin{equation}\label{map:pgestolge}
P(G,E,\phi_c)_\Sigma\to L(G,E,\phi_c).
\end{equation}

\begin{lem}\label{lem:pgestolge}
The algebra homomorphism \eqref{map:pgestolge} is an isomorphism. 
\end{lem}
\begin{proof}
The algebra $P(G,E,\phi_c)_{\Sigma}$ is obtained from $P(G,E,\phi_c)$ upon adjoining an element $y_e$ for each $e\in E^1$ so that $r(e)y_es(e)=y_e$ and so that the matrices
$M_v=\sum_{s(e)=v}\epsilon_{v,e}e$ and $N_v=\sum_{s(e),v}\epsilon_{e,v}y_e$ satisfy $M_vN_v=\epsilon_{v,v}v$ and $N_vM_v=\sum_{s(e)=v}\epsilon_{e,e}r(e)$. The homomorphism \eqref{map:pgestolge} is the inclusion on $P(G,E,\phi_c)$ and sends $y_e\mapsto e^*$; to prove it is an isomorphism it suffices to show that the $y_e$ satisfy the same relations \eqref{eq:cohn0}, \eqref{eq:cohn1} and \eqref{eq:cohn2} as the $e^*$. This is clear for all but the last identity of \eqref{eq:cohn2}. Moreover, we have 
\begin{gather*}
y_eg=y_egs(e)=y_eg\sum_{s(f)=s(e)}fy_f=y_e\sum_{s(f)=s(e)}g(f)\phi(g,f)y_f=\\
y_e\sum_{s(f)=s(e)}f\phi(g,g^{-1}f)y_{g^{-1}(f)}=r(e)\phi(g,g^{-1}e)y_{g^{-1}(e)}=\phi(g,g^{-1}e)y_{g^{-1}(e)}.
\end{gather*}
Thus the $y_e$  satisfy all the required identities for the existence of homomorphism of $P(G,E,\phi_c)$-algebras $C(G,E,\phi_c)\to P(G,E,\phi_c)_\Sigma$ mapping $e^*\mapsto y_e$. Furthermore, the identity $M_vN_v=\epsilon_{v,v}v$ implies that the latter induces a homomorphism $L(G,E,\phi_c)\to P(G,E,\phi_c)_\Sigma$ inverse to \eqref{map:pgestolge}.
\end{proof}

\begin{lem}\label{lem:pgetolgeflat}
The inclusion $P(G,E,\phi_c)\subset L(G,E,\phi_c)$ makes $L(G,E,\phi_c)$ into a flat left $P(G,E,\phi_c)$-module. 
\end{lem}
\begin{proof}
It follows from Lemma \ref{lem:pgestolge} that $P(G,E,\phi_c)\subset L(G,E,\phi_c)$ is a ring epimorphism. Hence by \cite{stenstrom}*{Theorem 2.1}, it suffices to find, for each $x\in L(G,E,\phi_c)$, a finite subset $\cF\subset \cP(E)$ such that 

\begin{gather}\label{condi1}
\forall \gamma\in\cF, \ \ x\gamma\in P(G,E,\phi_c)\\
\sum_{\gamma\in\cF}\gamma P(G,E,\phi_c)=P(G,E,\phi_c). \label{condi2}
\end{gather}
Any element $x\in L(G,E,\phi_c)$ can be written as a finite linear combination 
\begin{equation}\label{eq:pgetolgeflat}
x=\sum_{\alpha,\beta}\alpha x_{\alpha,\beta}\beta^*=\sum_{\beta}(\sum_{\alpha}\alpha x_{\alpha,\beta})\beta^*
\end{equation}
with each $x_{\alpha,\beta}\in R_{r(\alpha)}$ and $r(\beta)=r(\alpha)$. Let $\cF'\subset \cP(E)$ be the set of all those paths $\beta$ such that $\beta^*$ appears in \eqref{eq:pgetolgeflat} with a nonzero coefficient. Using $CK2$ we may arrange that there is an $n$ such all $\beta\in\cF'$ with $r(\beta)\in\sink(E)$ 
have length $\le n$, and all those with $r(\beta)\in\reg(E)$ have length $n$. Hence for all $\beta\in\cF'$
\[
x\beta=\sum_{\alpha}\alpha x_{\alpha,\beta}\in P(G,E,\phi_c).
\]
Let 
$$
\cP(E)\supset \cF=\{\gamma,\colon\, r(\gamma)\in\reg(E),\, \|\gamma|=n\}\cup
\{\gamma,\colon\, r(\gamma)\in\sink(E),\, \|\gamma|\le n\}
$$ 
then $\cF\supset\cF'$ and $x\gamma=0$ for all $\gamma\in\cF\setminus\cF'$. Hence $\cF$ satisfies \eqref{condi1}. Moreover $\sum_{\gamma\in\cF}\gamma\gamma^*=1$, so \eqref{condi2} is also satisied. 
\end{proof}

\begin{coro}\label{coro:pgetolgeflat}
If $E$ is finite and $P(G,E,\phi_c)$ is either right regular or right coherent, then the same is true of $L(G,E,\phi_c)$. 
\end{coro}
\begin{proof} By Lemmas \ref{lem:pgestolge} and \ref{lem:pgetolgeflat}, $P(G,E,\phi_c)\subset L(G,E,\phi_c)$ is a perfect right localization in the sense of \cite{stenstrom}*{Definition on page 229}. Hence by \cite{stenstrom}*{Theorem 2.1 (b)} the family $\fF$ of all right ideals $\fa\subset P(G,E,\phi_c)$ such that $\fa L(G,E,\phi_c)=L(G,E,\phi_c)$ is a Gabriel topology, and $L(G,E,\phi_c)=P(G,E,\phi_c)_\fF$ is the localization with respect to $\fF$. It then follows from \cite{stenstrom}*{Corollary 1.10 of Chapter IX and Proposition 3.4 of Chapter XI} that localization of right $P(G,E,\phi_c)$-modules with respect to $\fF$ is exact and essentially surjective. Moreover, it preserves projectivity by \cite{stenstrom}*{Proposition 1.11 of Chapter IX}. It follows that $L(G,E,\phi_c)$ is regular whenever $P(G,E,\phi_c)$ is. If $P(G,E,\phi_c)$ is right coherent then $L(G,E,\phi_c)$ is right coherent by \cite{stenstrom}*{Proposition 3.12 of Chapter XI}.
\end{proof}

The following two lemmas are probably well-known. I came to them together with my colleague Marco Farinati after a fruitful discussion. 

\begin{lem}\label{lem:hhreg}
Let $S$ be a unital ring containing a semisimple commutative ring $k$, and such that $S$ has left projective dimension $d$ as an $S^e:=S\otimes_kS^{\op}$-module. Then $S$ and $S^e$ have (both right and left) global projective dimensions $\le d$ and $\le 2d$, respectively. In particular, both $S$ and $S^e$ are regular. 
\end{lem}
\begin{proof}
Let $(S\otimes_k\bar{S}^{\otimes_k\bu}\otimes_kS,b')\fib S$ be the bar resolution. The hypothesis means that 
\[
\Omega^dS=\ker (b':S\otimes_k\bar{S}^{\otimes_k d-1}\otimes_kS\to S\otimes_k\bar{S}^{\otimes_k d-2}\otimes_kS)
\] 
is a projective $S\otimes_kS^{\op}$-module. Consider the truncated bar resolution
\[
Q_m=\left\{\begin{matrix} S\otimes_k \bar{S}^{\otimes_k m}\otimes_k S &\text{ if } 0\le m\le d-1\\
                          \Omega^dS &\text{ if } m=n\\
													 0 & \text{ if } m>d\end{matrix}\right.
\]
Observe that $Q_\bu$ is both right and left split; so tensoring it over $S$ on either side with an $S$-module $M$ yields a resolution $P\fib M$ of length $d$ such that each $P_m$ is a scalar extension of a $k$-module, and therefore projective, as $k$ is semisimple. Thus both the right and the left global dimension of $S$ are $\le d$. Next observe that $S^e$ is isomorphic to its opposite ring via the flip $s\otimes t\mapsto t\otimes s$; in particular its left and right global dimensions coincide. Let $M$ be a left $S^e$-module; then $Q_\bu\otimes_SM\otimes_SQ_\bu$ is the total complex of a bicomplex whose $m$-th row is an $S^e$-projective resolution of $M\otimes_SQ_m$. Hence the composite $Q_\bu\otimes_SM\otimes_SQ_\bu\fib M\otimes_SQ_\bu\fib M$ is a quasi-isomorphism. This completes the proof that the left global dimension of $S^e$ is $\le 2d$, since $Q_\bu\otimes_SM\otimes_SQ_\bu$ has length $2d$. 
 \end{proof}

\begin{lem}\label{lem:treg}
Let $S$ and $k$ be as in Lemma \ref{lem:hhreg} and let $M$ be a left $S\otimes_kS^{\op}$-module. If either $M_S$ or ${}_SM$ is projective, then the tensor algebra $T=T_S(M)$ has both right and left global dimension $\le 2d+1$. In particular, $T$ is regular. 
\end{lem}
\begin{proof}
We shall assume that $M$ is right projective. In view of Lemma \ref{lem:hhreg} it suffices to show that the $T\otimes_kT^{\op}$-module $T$ has projective dimension $\le 2d$. Let $\Omega^\bu$ be as in the proof of Lemma \ref{lem:hhreg}. Consider the relative cotangent sequence \cite{cq1}*{Corollary 2.10}
\begin{equation}\label{seq:cotan}
0\to T\otimes_S\Omega^1S\otimes_ST\to \Omega^1T\to T\otimes_SM\otimes_ST\to 0.
\end{equation}
It suffices to show that the bimodules on the left and right of the sequence above have projective dimension $\le 2d$.  
An appropriate truncation of the bar resolution provides an $S\otimes_kS^{\op}$-projective resolution $P_\bu\fib\Omega^1_S$ of length $\le d-1$  which is both right and left split. Hence  $P_\bu\otimes_ST\to \Omega^1S\otimes_ST$ is a quasi-isomorphism because $P_\bu$ is right split, and  $T\otimes_S P_\bu\otimes_S T\to T\otimes_S\Omega^1S\otimes_ST$ is a quasi-isomorphism because $T$ is right projective. Moreover $T\otimes_S P_\bu\otimes_S T=T\otimes_kT^{\op}\otimes_{S\otimes_kS^{\op}}P_\bu$ is projective because $P_\bu$ is. So the first term from the left in the exact sequence \eqref{seq:cotan} has projective dimension at most $d-1$.  Next we consider the last term of \eqref{seq:cotan}. By Lemma \ref{lem:hhreg}, there is an $S^e$-projective resolution
$P'_\bu\fib M$ of length $\le 2d$, which is split both as a complex of right and of left modules. Then $P'_\bu\otimes_ST\to M\otimes_ST$ is quasi-isomorphism, and therefore so is $T\otimes_S P'_\bu\otimes_ST\to T\otimes_SM\otimes_ST$, since $T_S$ is projective. This concludes the proof.
\end{proof}

\begin{prop}\label{prop:lgereg} Let $(G,E,\phi_c)$ be a twisted EP tuple. Assume that $E$ is finite and regular and that $G$ acts trivially on $E^0$. Further assume that for every $v\in E^0$, $R_v$ contains a field $k_v$ and has finite projective dimension as a left $R_v\otimes_kR_v^{\op}$-module. Then $L(G,E,\phi_c)$ is a regular ring. 
\end{prop}
\begin{proof} Let $k=\bigoplus_{v\in E^0}k_v$; then $k$ is semisimple, and the hypothesis implies that $R$ has finite projective dimension as an $R\otimes_kR^{\op}$-module. Since $X$ is right-projective, Lemma \ref{lem:treg} tells us that $P(G,E,\phi_c)$ is regular. Hence $L(G,E,\phi_c)$ is regular by Corollary \ref{coro:pgetolgeflat}.  
\end{proof}

\begin{coro}\label{coro:lgereg}
Let $\ell$ be a field and let $(A,B,C)$ be a twisted Katsura triple. Then $\cO_{A,B}^C$ is a regular ring.
\end{coro}
\begin{proof}
Let $v\in E^0$ and let $L=\ell[t,t^{-1}]$. As shown in the proof of Lemma \ref{lem:flat2}, $R_v$ is either $L$ or a product of copies of $\ell$. Thus $R_v$ satisfies the hypothesis of Proposition \ref{prop:lgereg}.
\end{proof}

\end{document}